\documentclass[12pt]{amsart}

\title[Existence of concentration optimizers for the Wigner distribution]{On the existence of optimizers for nonlinear time-frequency concentration problems: \\ the Wigner distribution}

\author{Federico Stra}
\address{Politecnico di Torino, Dipartimento di Scienze Matematiche ``G. L. Lagrange'', corso Duca degli Abruzzi 24, 10129 Torino, Italy}
\email{federico.stra@polito.it}

\author{Erling Svela}
\address{Department of Mathematical Sciences, Norwegian University of Science and Technology, 7491 Trondheim, Norway}
\email{erling.a.t.svela@ntnu.no}

\author{S. Ivan Trapasso}
\address{Politecnico di Torino, Dipartimento di Scienze Matematiche ``G. L. Lagrange'', corso Duca degli Abruzzi 24, 10129 Torino, Italy}
\email{salvatoreivan.trapasso@polito.it}

\usepackage{graphicx}

\usepackage{amsthm}
\usepackage{amssymb}
\usepackage{amsmath}
\usepackage{xcolor}
 \usepackage{physics}

\usepackage{hyperref}
\usepackage{appendix}

\usepackage{mathtools}
\mathtoolsset{showonlyrefs}

\usepackage{nicematrix,enumitem}

\calclayout

\usepackage{geometry}
\usepackage{microtype}

\newtheorem{theorem}{Theorem}[section]
\newtheorem{prop}[theorem]{Proposition}
\newtheorem{corollary}[theorem]{Corollary}
\newtheorem{lemma}[theorem]{Lemma} 

\theoremstyle{definition}

\newtheorem{question}{Question}

\theoremstyle{remark}
\newtheorem{remark}[theorem]{Remark} 

\usepackage{cite}
\usepackage{verbatim}

\newcommand{\C}{\mathbb{C}}
\newcommand{\R}{\mathbb{R}}

\newcommand{\Z}{\mathbb{Z}}
\newcommand{\N}{\mathbb{N}}
\newcommand{\T}{\mathbb{T}}

\newcommand{\cC}{\mathcal{C}}

\DeclareFontFamily{U}{mathx}{}
\DeclareFontShape{U}{mathx}{m}{n}{<-> mathx10}{}
\DeclareSymbolFont{mathx}{U}{mathx}{m}{n}
\DeclareMathAccent{\widehat}{0}{mathx}{"70}
\DeclareMathAccent{\widecheck}{0}{mathx}{"71}

\newcommand{\exn}[1]{#1^{(n)}}
\def\fn{f^{(n)}}

\def\Fkn{F_k^{(n)}}

\def\xn{x^{(n)}}
\def\xin{\xi^{(n)}}
\def\an{a^{(n)}}
\def\bn{b^{(n)}}
\def\cn{c^{(n)}}
\def\dn{d^{(n)}}
\def\den{\delta^{(n)}}
\def\zjn{z_j^{(n)}}
\def\zjpn{z_{j'}^{(n)}}
\def\zln{z_{\ell}^{(n)}}

\def\wkn{w_k^{(n)}}

\def\cJ{\mathcal{J}}

\def\cS{\mathcal{S}}

\def\rd{\mathbb{R}^d}
\def\rdd{\mathbb{R}^{2d}}

\renewcommand{\dd}[1]{\,\mathrm{d} #1}   
\renewcommand{\Re}{\mathrm{Re}}

\newcommand{\cSjn}{\cS_{j,j'}^{(n)}}

\def\cnj{c^{(n)}_{j,j'}}
\def\cnjp{c^{(n)}_{j',j}}

\def\dnj{d^{(n)}_{j,j'}}

\newcommand{\BJ}{W_{\mathrm{BJ}}}

\begin{document}

\begin{abstract}
We prove that, for any measurable phase space subset $\Omega\subset\mathbb{R}^{2d}$ with $0<|\Omega|<\infty$ and any $1\le p < \infty$, the nonlinear concentration problem
\[
\sup_{f \in L^2(\mathbb{R}^d)\setminus\{0\}}\frac{\|Wf\|_{L^p(\Omega)}}{\|f\|_{L^2}^2}
\]
admits an optimizer, where $Wf$ is the Wigner distribution of $f$. The main obstruction is that $Wf$ is covariant (not invariant) under time-frequency shifts, which impedes weak upper semicontinuity, so the effects of constructive interference must be taken into account. We close this compactness gap via concentration compactness for Heisenberg-type dislocations, together with a new asymptotic formula that quantifies the limiting contribution to concentration over $\Omega$ from asymptotically separated wave packets. When $p=\infty$ we also identify the sharp constant $2^d$ and show that it is attained.

We also discuss some related extensions: for $\tau$-Wigner distributions with $\tau \in (0,1)$ we isolate a chain phenomenon that obstructs the same strategy beyond the Wigner case ($\tau=1/2$), while for the Born--Jordan distribution in $d=1$ we obtain weak continuity, and thus existence of concentration optimizers for all $1\le p<\infty$ (the $p=\infty$ supremum equals $\pi$ but is not attained).
\end{abstract}

\subjclass[2020]{49Q10, 49R05, 42B10, 94A12, 81S30}
\keywords{Nonlinear time-frequency concentration, optimization, Wigner distribution, Born-Jordan distribution, concentration compactness}

\maketitle

\section{Introduction}

Though approaching its centennial, the Wigner distribution is still one of the most popular time-frequency representations in signal analysis, not to mention its prominent role in many problems and applications of quantum mechanics. Introduced somewhat enigmatically by Wigner in 1932 as a quasi-probability distribution on phase space to account for quantum corrections to classical statistical mechanics~\cite{Wigner32}, it was later rediscovered by Ville in 1948 \cite{ville}. Since then, it has become a widely used time-frequency method in signal processing due to its robustness and desirable features. 

To be specific, recall that the Wigner distribution of \(f\in L^2(\R^d)\) is defined by \begin{align*}
    Wf(z)=\int_{\R^d} e^{-2\pi i \xi \cdot y} f \Big(x+\frac{y}{2} \Big)\overline{f\Big(x-\frac{y}{2}\Big)} \ \dd y, \qquad z=(x,\xi)\in \rdd.
\end{align*}
Therefore, $Wf$ coincides with the Fourier transform of the two-point autocorrelation of $f$. Informally, we may also view $Wf(x,\xi)$ as the joint time-frequency energy density of $f$ localized near the phase space point $(x,\xi)$. We note that the map $f \mapsto Wf$ is inherently quadratic: For every $\alpha,\beta \in \C$ and $f,g \in L^2(\rd)$ we have
\[ W(\alpha f+\beta g) = \abs{\alpha}^2 Wf + \alpha\bar{\beta} W(f,g) + \bar{\alpha}\beta W(g,f) + \abs{\beta}^2 Wg, \]
where $W(f,g)$ and $W(g,f)$ denote cross-Wigner distributions, cf.\ \eqref{eq-cwig} below. The occurrence of these cross-interference terms, along with self-interference fringes, is a classical drawback in signal processing: Even for well-localized signals (such as Gaussian functions) the joint time-frequency energy exhibits non-trivial scattering phenomena in phase space~\cite{BDDO,cohen1995time,cohen1995uncertainty}. While the uncertainty principle prevents \(Wf\) from being localized within a domain \(\Omega \subset \rdd\) of arbitrarily small measure \cite{GrochenigUP}, it is natural to investigate whether optimal phase space concentration on a given domain can be achieved despite these interferences and obstructions --- a problem of both theoretical and practical interest, see for instance~\cite{LiebOstrover,RT1,Flandrin,BDW99,JanssenSurvey} and the references therein. 

There are several possible choices for how to measure phase space concentration. A convenient one is the $L^2$ norm of $Wf$ over $\Omega$, which has to be interpreted as the fraction of time-frequency energy trapped within $\Omega$. More generally, one can use the $L^p$ norms with $p \ne 2$, which can be rightfully viewed as favoring spreading or concentration/sparsity depending on the range of $p$. Globally, this heuristic point of view is reinforced by the $L^p$-norm inequalities due to Lieb~\cite{Lieb}: 
\[ 
\begin{cases}
    \norm{Wf}_{L^p} \ge C_{p,d} \norm{f}_{L^2}^2 & (1 \le p \le 2) \\ 
    \norm{Wf}_{L^p} \le C_{p,d} \norm{f}_{L^2}^2 & (p\ge 2), 
\end{cases} \qquad C_{p,d} = \Big(\frac{2^{p-1}}{p} \Big)^{d/p}.
\]
Nevertheless, when it comes to local versions, it has remained open whether $L^p$ concentration can be optimized over a given domain of finite, positive measure. This is precisely the problem addressed in this note, and our main result settles it in the affirmative. 

\begin{theorem}\label{BigTheorem}
    Let \(\Omega\subset\R^{2d}\) be measurable with \(0<|\Omega|<\infty\) and let \(p\in [1,\infty)\). The supremum\begin{align}\label{eq-WigSup}
        L=\sup_{f\in L^2(\R^d)\setminus\{0\}}\frac{\left(\int_{\Omega}|Wf(z)|^p \dd z\right)^{1/p}}{\|f\|_{L^2}^2}
    \end{align}
    is attained. 
\end{theorem} \noindent The same conclusion holds in the case where $p=\infty$, the proof being considerably easier --- see Section~\ref{sec-linfty} below, where we also find the optimal value $L=2^d$. 

Let us highlight that in the linear case (that is, optimizing the integral \(\int_{\Omega} Wf(z) \dd z\) --- see the recent monograph \cite{lerner} for interesting related problems) a minimax argument reduces the concentration problem to the analysis of a certain compact operator, see for instance~\cite{Flandrin,LiebOstrover,DaubechiesExt}. On the other hand, the nonlinear nature of the problem \eqref{eq-WigSup} precludes such linear techniques and, in spite of its simple formulation, requires sophisticated tools. Indeed, our proof of Theorem~\ref{BigTheorem} relies on the theory of \textit{concentration compactness} \cite{ConcCompactnessBook,Lions1,Lions2}, a powerful framework especially designed to handle loss of compactness due to invariance of the target functional under the action of a non-compact family of transformations (the so-called \textit{dislocations}). Before giving more details, let us stress that this approach has been successfully used in time-frequency concentration problems only recently: The authors of \cite{AmbConcentration} solved with this approach the exact analogue of the problem in this note for the \textit{ambiguity function}, that is
\[    Af(z)=\int_{\R^d} e^{-2\pi i \xi \cdot y} f \Big(y+\frac{x}{2} \Big)\overline{f\Big(y-\frac{x}{2}\Big)} \ \dd y, \qquad z=(x,\xi)\in \rdd. \]

While the optimization objective in \eqref{eq-WigSup} is closely related to the one considered in \cite{AmbConcentration} (ultimately by symplectic Fourier transform), the Wigner distribution and the ambiguity function behave quite differently in some respects. Concretely, recall that the time-frequency shift of $f \in L^2(\rd)$ by $z=(x,\xi) \in \rdd$ is defined by
\[ \pi(z)f(y) = e^{2\pi i \xi \cdot y} f(y-x). \]
The ambiguity function is completely invariant under time-frequency shifts, that is $\abs{A(\pi(z)f)} = \abs{Af}$, hence the corresponding concentration functional \eqref{eq-WigSup} is invariant as well. These remarks clearly motivate resorting to concentration compactness, as the main defect of compactness in this problem comes from the non-compact symmetry $f\mapsto \pi(z)f$, which allows phase-space mass to drift away without changing the objective. To keep track of the energy that might otherwise be lost at infinity we perform a \textit{profile decomposition}: Up to subsequences, a maximizing sequence $f^{(n)}$ can be split into a finite sum of separated packets as
\[
f^{(n)}=\sum_{j=1}^{k}\pi(z^{(n)}_j)f_j \;+\; w_k^{(n)}, \qquad k \in \N\setminus\{0\},
\]
where $f_1,\ldots,f_k$ are fixed \textit{profiles} whose centers $z^{(n)}_j$ separate asymptotically (i.e., $|z^{(n)}_j-z^{(n)}_{j'}|\to\infty$ for $j\neq j'$) and the remainder $w_k^{(n)}$ is small in a suitable sense --- see Lemma~\ref{lem-profdec} below for a more precise account on the matter.

On the other hand, the behavior of the Wigner transform under time-frequency shift is subtler: it is only \textit{covariant} under time-frequency shifts, that is $\abs{W(\pi(z)f)(w)} = \abs{Wf (w -z)}$ (cf.\ \eqref{eq-covwig} below). It is therefore not evident whether a concentration compactness strategy could lead to useful conclusions in this context. To shed some light on the issue, we emphasize that the optimization is effectively translation-invariant, since
\begin{equation*}
  L
  =\sup_{f \in L^2(\rd)}\frac{\|Wf\|_{L^p(\Omega)}}{\|f\|_{L^2}^2}
  =\sup_{z\in\rdd}\sup_{f \in L^2(\rd)}\frac{\|\,Wf(\cdot-z)\,\|_{L^p(\Omega)}}{\|f\|_{L^2}^2},
\end{equation*}  and that the target functional is invariant under antipodal pairing, that is \[\abs{W(\pi(z)f,\pi(-z)f)} = \abs{Wf}.\] Nevertheless, loss of compactness still stems from mass escaping to infinity in phase space as before, but the decisive difference is that, unlike the ambiguity case, Wigner cross-terms may survive in the limit when \textit{midpoints} of profiles remain bounded, producing potentially constructive interference from antipodally shifted profiles. Moreover, as detailed below, this invariance causes the failure of semicontinuity in the optimization problem, which further justifies attempts beyond the direct methods, such as resorting to concentration compactness. A crucial component of our analysis consists in isolating these \textit{surviving pairs} in a profile decomposition and quantifying their limiting contribution to the $L^p(\Omega)$ norm (see Section~\ref{sec-crosst}), which is the key extra ingredient (of independent interest) needed to make the concentration compactness approach work for the Wigner concentration problem. 

This leaves open the possibility of \textit{exotic optimizers} when surviving pairs, despite being arbitrarily far apart in phase space in the limit, interfere constructively to optimize concentration on $\Omega$. Although we cannot rule out the occurrence of such a degenerate scenario, we show that it is exceptional in a sense detailed in Section~\ref{sec:Exotic}. 

The next natural step beyond the Wigner transform is considering the analogous nonlinear concentration problem for more general covariant time-frequency representations. A large and popular family of those stems from averaging the Wigner distribution with a suitable kernel: the so-called \textit{Cohen class} consists of distributions of the form
\[  
Qf = Wf*\sigma, \qquad \sigma \in \cS'(\rdd). 
\]
Phase space averaging naturally promotes smoothing, which in turn is expected to better tame interferences (hence, degenerate scenarios) in concentration problems. We are able to confirm this intuition for several types of kernels, but proofs involve a completely different set of techniques than those used in this paper (that is, the recent advances in quantum harmonic analysis~\cite{Bible1,Tauberian,QTFA}), and this is the reason why we prefer to postpone these studies to a separate, forthcoming manuscript. 

On the other hand, such heuristics do not make a general rule: In Section~\ref{sec-gen} we discuss some findings for a simple one-parameter generalization of the Wigner distribution, that is the $\tau$-Wigner distribution: 
\[ W_{\tau}f(z)=\int_{\R^d}e^{-2\pi i \xi \cdot y} f(x+\tau y)\overline{f(x-(1-\tau)y)} \; \dd y, \qquad \tau \in (0,1). \]
These are members of the Cohen class, since for $\tau \ne 1/2$ they arise from averaging the Wigner distribution with a chirp-like kernel in the Fourier domain: 
\[ W_\tau f = Wf * \sigma_\tau, \qquad \widehat{\sigma_\tau}(x,\xi) = e^{-2\pi i (\tau-1/2) x\cdot\xi}. \]
Despite appearances, when $\tau \ne 1/2$ the structure of surviving cross terms is entirely different from that of the Wigner case. We still attempt to push the same strategy as far as possible, and then pinpoint the currently open gap --- which is again related to quantification of asymptotic contributions, where now pairs of escaping profiles may also exhibit infinite chaining phenomena. 

To mitigate the asymmetry with respect to $\tau$ beyond the Wigner case we employ an additional averaging strategy, originating in signal processing to further tame the effect of interferences for Wigner-type distributions~\cite{Hlawatsch}. Indeed, the \textit{Born--Jordan distribution} can be defined equivalently by 
\[ \BJ f(z)\,=\,\int_0^1 W_\tau f(z)\, \dd \tau = Wf* \Theta_\sigma, 
	\qquad 
	\Theta(\zeta)=\frac{\sin(\pi\,\zeta_1\!\cdot\!\zeta_2)}{\pi\,\zeta_1\!\cdot\!\zeta_2}, \quad \zeta=(\zeta_1,\zeta_2) \in \rdd, 
 \]

 where $\Theta_\sigma$ stands for the symplectic Fourier transform of $\Theta$. The Born--Jordan kernel is of very low regularity, as it fails to be locally in any $L^p$ space --- see for instance \cite{CGT_18,DG_book_BJ} and the references therein, also for framing its relevance in quantum mechanics. Nevertheless, at least in the planar case (i.e., $d=1$) we are able to confirm the smoothing effect of integration over $\tau$ and obtain a full affirmative answer to the concentration problem (see Section~\ref{ssec:BornJordan}), although for $p=\infty$ the supremum is not attained. In fact, the profile decomposition machinery actually shows that the concentration functional corresponding to the Born--Jordan distribution is weakly continuous on $L^2(\rd)$. This fact is particularly interesting since it is a nontrivial\footnote{For a fixed \(g\in L^2(\rd)\) the quadratic map $f \mapsto |A(f,g)|^2 = Wf*W\check{g}$, known as the spectrogram of \(f\), is also sequentially weakly continuous from $L^2(\rd)$ to $L^p(\Omega)$, for $1\leq p<\infty$. However, its continuity is a direct consequence of the analogous sequential weak continuity of the linear map $f \mapsto A(f,g)$ computing the ambiguity transform, which was proved in \cite[Proposition~5.1]{AmbConcentration}.} instance of a $L^p(\Omega)$-concentration functional associated with a quadratic time-frequency distribution where (semi)continuity does not fail. Moreover, this novel evidence is in line with known heuristics from signal processing about the strong smoothing effect of the Born--Jordan kernel, but here it is rigorously confirmed at a deeper topological level in terms of a regularity upgrade with respect to the analogous problem for the Wigner distribution. The current techniques do not allow us to extend the conclusion beyond the one-dimensional case, but it is actually unclear whether the supremum is finite when $d>1$. Indeed, the nodal set of the kernel is a considerably larger manifold in that case. Perhaps restricting the optimization environment to a subspace of $L^2(\rd)$ consisting of more regular functions could help (e.g., the Feichtinger algebra $M^1(\rd)$ or other modulation spaces \cite{Fei,feichtinger1983modulation}), but these questions are left to future investigations.\newline 

The remainder of the paper is organized as follows. We collect the required preliminaries in Section~\ref{sec-prel}, while the technical results needed to prove Theorem~\ref{BigTheorem} are proven in Section~\ref{sec-crosst}, where we also discuss some consequences, such as failure of semicontinuity. In Section~\ref{sec-proofbig} we give the proof of Theorem~\ref{BigTheorem}. The concentration problem in the case \(p=\infty\) is settled in Section~\ref{sec-linfty}. We conclude with the analysis of the concentration problem for the \(\tau\)-Wigner and Born--Jordan distributions in Section~\ref{sec-gen}.

\section{Preliminaries}\label{sec-prel}

\subsection{Notation} We write \(a\lesssim b\) if there is a universal constant \(C>0\) such that \(a\leq C\ b\). If the constant depends on an auxiliary parameter \(\lambda\), that is \(C=C(\lambda)\), then we write \(a\lesssim_\lambda b\).

The inner product on $L^2(\rd)$ is assumed to be conjugate-linear in the second argument, that is $\langle f,g \rangle = \int_{\rd} f(y) \overline{g(y)} \dd y$. The same brackets are used to denote the duality between a temperate distribution $f \in \cS'(\rd)$ and a function $g$ in the Schwartz class. 

If a sequence \(\{f^{(n)}\}\) in \(L^2(\rd)\) converges weakly to \(f\) we write \(f^{(n)}\rightharpoonup f.\)

The Fourier transform of $f \colon \rd \to \C$, whenever meaningful, is $\widehat{f}(\omega) = \int_{\rd} e^{-2\pi i \omega \cdot y} f(y) \dd y$. 

For $\Omega \subset \rdd$ we denote by \(\|\cdot\|_{L^p(\Omega)}\) the local \(p\)-norm on \(\Omega\), while \(\|\cdot\|_{L^p}\) always denotes the global \(p\)-norm over $\rdd$. 

\subsection{Some results in time-frequency analysis}

A recurring theme in the treatment of quadratic time-frequency representations is the (unavoidable) appearance of cross terms. In our context, this crucial role is played by the \textit{cross-Wigner distribution}: Given $f,g\in L^2(\rd)$, we set
\begin{equation} \label{eq-cwig}
    W(f,g)(z)=\int_{\R^d} e^{-2\pi i \xi \cdot y} f\left(x+\frac{y}{2} \right)\overline{g\left(x-\frac{y}{2}\right)} \;\dd y, \qquad z=(x,\xi)\in\rdd. 
\end{equation}
Let us now collect some important properties of the cross-Wigner distribution, see~\cite[Section 4]{Grochenig} for details. Letting \(g=f\) recovers $Wf=W(f,f)$. We have that \(\overline{W(f,g)}=W(g,f)\), and in particular \(Wf\) is real. 

A simple density argument, taking $g \in \cS(\rd)$ first, shows that $W(f,g) \in L^2(\rdd)$ for all $f,g \in L^2(\rd)$, and the following orthogonality relation (Moyal's identity) holds for \(f_1,f_2,g_1,g_2\in L^2(\rd)\):
\begin{equation*}  \langle W(f_1,g_1),W(f_2,g_2)\rangle_{L^2(\rdd)}=\langle f_1,f_2\rangle_{L^2(\rd)}\overline{\langle g_1,g_2\rangle_{L^2(\rd)}}. \end{equation*}
As a result, and by a straightforward application of Cauchy-Schwarz's inequality, we obtain the norm bounds
\begin{equation} \label{eq-wignorm-bds}
    \norm{W(f,g)}_{L^2} = \|f\|_{L^2}\|g\|_{L^2}, \qquad \|W(f,g)\|_{L^\infty}\leq 2^d\|f\|_{L^2}\|g\|_{L^2}.
\end{equation}
Furthermore, it can be proved that $W(f,g)$ belongs to the space $C_0(\rdd)$ of continuous functions on $\rdd$ that vanish at infinity.

A crucial property for our purposes is the \textit{covariance under time-frequency shifts}. Recall that, given \(z=(x,\xi)\in \rdd\), the time-frequency shift \(\pi(z)\) acts on \(f\in L^2(\rd)\) by \(\pi(z)f(y)=e^{2\pi i \xi \cdot y} f(y-x)\). 
Setting \([z,u]\) for the symplectic inner product of \(z\) and \(u\), that is 
\[ [z,u]=x_z\cdot\xi_u-\xi_z\cdot x_u=z\cdot Ju, \qquad J=\begin{pmatrix}
    O & I \\
    -I & O
\end{pmatrix}\in \R^{2d\times 2d}, \] the behavior of the Wigner distribution under time-frequency shifts reads as follows.

\begin{lemma}[\!\!{\cite[Proposition 3.1]{CNT_19}}] \label{lem-covwig}
	Let $a=(x_a,\xi_a),\ b=(x_b,\xi_b) \in \rdd$ and set
	\begin{equation*} 
	c(a,b)
	\coloneqq \frac{1}{2}\big(x_a+ x_b\,,\,\xi_a+\xi_b\big),
	\end{equation*}
	Then, for all $f,g \in L^2(\rd)$ and $z \in \rdd$, 
		\begin{equation}
			W(\pi(a)f,\pi(b)g)(z)
			= e^{\pi i (\xi_a+\xi_b)\cdot(x_a-x_b)}e^{2\pi i [z,a-b]} \, W(f,g)\!\big(z-c(a,b)\big).
			\label{eq-covwig}
		\end{equation}
\end{lemma}

Refined continuity properties of the mapping \(f\to Wf\) can be given for functions with refined time-frequency regularity, such as those belonging to modulation and Wiener amalgam spaces. To avoid introducing auxiliary notions we have chosen to follow the approach of~\cite[Section 17]{deGossonBook}, which is equivalent to the standard one via short-time Fourier transform up to harmless global constants. For a thorough treatment of modulation spaces, as well as their standard definition, we direct the reader to~\cite{feichtinger1983modulation,Feichtinger2006}.

We say that a temperate distribution \(f\in \mathcal{S}'(\rd)\) belongs to the modulation space \(M^{\infty}(\rd)\) if for some (in fact, any) $g \in \cS(\rd)\setminus\{0\}$ we have \begin{align*}
    \|f\|_{M^{\infty}}\coloneqq \sup_{z\in\rdd} \abs{W(f,g)(z)}<\infty. 
\end{align*}
Different choices of $g$ result in equivalent norms on the Banach space $M^\infty(\rd)$. We also stress the continuous inclusion $L^2(\rd) \hookrightarrow M^\infty(\rd)$ in view of \eqref{eq-wignorm-bds}. 

We also need to deal with a special Wiener amalgam space~\cite{FournierStewart}: We say that a measurable function \(f \colon \rd\to\C\) belongs to the Wiener amalgam space \(W(L^2,L^\infty)(\rd)\) if for some (hence any) function $g \in C^\infty(\rd)\setminus\{0\}$ with compact support we have
\begin{equation}
    \|f\|_{W(L^2,L^\infty)}\coloneqq \norm{ \norm{ f(y) g(y-x)}_{L^2_y}}_{L^\infty_x}<\infty.
\end{equation}
See also ~\cite{FL06,Fournier,Heil}. Intuitively speaking, we may think of a function in \(W(L^2,L^\infty)(\rd)\) as having simultaneously local \(L^2\) and global \(L^\infty\) regularity. As a concrete instance of this viewpoint, it is not difficult to show that if \(K\) is a compact subset of \(\R^d\) then 
\begin{equation} \label{eq-WL2Linfty-L2K}
    \|f\|_{L^2(K)}\lesssim_K\|f\|_{W(L^2,L^\infty)}\lesssim \|f\|_{L^2},
\end{equation}
and therefore the continuous inclusions \(L^2(\rd) \hookrightarrow W(L^2,L^\infty)(\R^d) \hookrightarrow L^2(K)\) hold.

A key part of our argument involves refining the bounds \eqref{eq-wignorm-bds} for the cross-Wigner distribution \(W(f,g)\) in terms of the norms just introduced. 

\begin{theorem}[\!\!{\cite[Theorem 4.12]{CT_20}}]\label{thm-embed}
    Let \(f\in L^{2}(\rd)\) and \(g\in M^{\infty}(\rd)\). Then \(W(f,g)\in W(L^2,L^\infty)(\rdd)\), with
		\begin{align*}
			\|W (f,g)\|_{W(L^2,L^\infty)}\lesssim_{d} \|f\|_{L^2}\|g\|_{M^{\infty}}.
		\end{align*}
	\end{theorem}

\subsection{Profile decomposition} We will use profile decomposition in \(L^2(\R^d)\), using the time-frequency shifts \(\{\pi(z)\}_{z\in \R^{2d}}\) as the set of dislocations. The details of this decomposition can be found in~\cite{AmbConcentration}, where this technique was introduced in order to study a similar problem. The defining properties are recalled below. For an account of the general theory, see~\cite{Lions1,Lions2,ConcCompactnessBook}. Here we confine ourselves to collecting the relevant facts in the following result. 

\begin{lemma} \label{lem-profdec}
    Let \(\fn\) be a sequence in \(L^2(\R^d)\) such that \(\displaystyle \limsup_{n\to \infty} \|\fn\|_{L^2}\leq 1\). There exists a subsequence (still denoted with $\fn$), profiles \(f_j\in L^2(\R^d)\) and points \( \zjn \in \R^{2d}\), $j \in \N$, such that for every integer \(k \ge 1\) we have 
    \begin{align*}
        \fn=\sum_{j=1}^k\pi(\zjn)f_j+\wkn,
    \end{align*} where \(\wkn \in L^2(\R^d)\) and the following conditions hold: 
    
    \begin{itemize}
        \item $\displaystyle\lim_{n\to \infty} |\zjn - \zjpn|=+\infty$ whenever \( j \neq j' \), \( f_j\ne 0\) and \(f_{j'}\neq 0\). 
        \item $\displaystyle \sum_{j=1}^k \|f_j\|_{L^2}^2+\displaystyle\limsup_{n\to\infty}\|\wkn\|_{L^2}^2\leq 1$.
        \item $\displaystyle \lim_{k\to \infty}\limsup_{n\to \infty} \|\wkn\|_{M^\infty}=0$.
        \item For each \(k\geq 1\) and \(1\leq j \leq k\) we have \(\displaystyle\lim_{n\to \infty} \pi(\zjn)^*\wkn=0\) in the weak topology on \(L^2(\R^d)\).
        \item $\displaystyle \lim_{n\to \infty}\Big\|\sum_{j=1}^k\pi(\zjn)f_j \Big\|_{L^2}^2=\sum_{j=1}^k\|f_j\|_{L^2}^2$. 
    \end{itemize}
\end{lemma}

\section{Controlling cross terms} \label{sec-crosst}
    \subsection{An asymptotic formula}
	The following result will play a central role in taming the cross terms arising from profile decompositions for our optimization problem. It is in fact a result of independent interest, since it provides a quantitative control on the $L^p$ concentration on a given phase space domain asymptotically generated by constructive interference of antipodally escaping wave packets. 
	
	\begin{theorem} \label{thm-avgbd}
		Let $f,g\in L^2(\R^d)$, and let $a^{(n)},b^{(n)}\in\R^{2d}$ satisfy
		\[
		d^{(n)} \coloneqq a^{(n)}-b^{(n)},\quad
        |d^{(n)}| \to\infty
        ,\qquad
		c^{(n)} \coloneqq \frac{a^{(n)}+b^{(n)}}{2}\to c\in\R^{2d}.
		\]
		Define
		\[
		\mathcal{S}_n(z) \coloneqq W(\pi(a^{(n)})f,\pi(b^{(n)})g)(z)+W(\pi(b^{(n)})g,\pi(a^{(n)})f)(z).
		\]
		For every measurable $\Omega\subset\R^{2d}$ with $0<|\Omega|<\infty$ and $1\le p<\infty$,
		\begin{equation} \label{eq-polar}
			\lim_{n\to\infty}\ \|\mathcal{S}_n\|_{L^p(\Omega)}
			 =
				2\,\cC_p\,\|W(f,g)\|_{L^p(\Omega-c)}\, 
		\end{equation}
        and
        \begin{equation}  \label{eq-polar2}
            2\,\cC_p\,\|W(f,g)\|_{L^p(\Omega-c)}\, 
        \le 
            L (\norm{f}_{L^2}^2 + \norm{g}_{L^2}^2),
        \end{equation}
        where $L$ is defined in \eqref{eq-WigSup} and
		\begin{equation} \label{eq-Cp}
		\mathcal{C}_p=\left( \frac{1}{2\pi} \int_0^{2\pi} \abs{\cos(\theta)}^p \dd \theta\right)^{1/p} = \Big(\frac{1}{\pi} \mathrm{B}\Big(\frac{p+1}{2},\frac12\Big)\Big)^{1/p},
		\end{equation} where $\mathrm{B}$ is Euler's beta function. \end{theorem}
\begin{remark}
    When \(p<2\), \(\mathcal{S}_n\) need not be in \(L^p(\R^{2d})\). However, each cross-Wigner term belongs to $L^2(\mathbb{R}^{2d})$, hence also to $L^2(\Omega)\hookrightarrow L^p(\Omega)$ for every $1 \le p \le 2$. As a result, by triangle inequality we have \(\mathcal{S}_n\in L^p(\Omega)\).
\end{remark}
Before delving into the details of the proof, some comments on the heuristics behind this result could be helpful. Setting $\an=(\xn_a,\xin_a)$ and $\bn=(\xn_b,\xin_b)$, we have by covariance (Lemma~\ref{lem-covwig}): 
\begin{align*}
			W(\pi(\an)f,\pi(\bn)g)(z)=e^{\pi i (\xin_a+\xin_b)\cdot(\xn_a-\xn_b)}e^{2\pi i [z,\dn]}W(f,g)(z-\cn).
		\end{align*} Since \(W(g,f)=\overline{W(f,g)}\), after setting \(\den=2\pi\dn\) and \(\exn{\phi}=\pi  (\xin_a+\xin_b)\cdot(\xn_a-\xn_b) + [\cn,\den]\) we get 
\begin{equation}
	\mathcal{S}_n(z)=2\Re(e^{i(\exn{\phi} +[z-\cn,\den])} W(f,g)(z-\cn)). 
\end{equation} In Equation~\eqref{eq-polar}, the symmetric block $\mathcal{S}_n(z)$ accounting for the cross-Wigner interaction between two wave packets can then be viewed as the (the real part of the) product of a carrier plane wave with frequency $\den$ and an envelope $W(f,g)$, both recentered at $\cn$. In fact, expanding the real part after letting \(\alpha(z)\) denote the phase of \(W(f,g)(z)\), we obtain an explicit cosine grating of the envelope: 
\[ \mathcal{S}_n(z) = 2 \abs{W(f,g)(z-\cn)} \cos(\exn{\phi} +[z-\cn,\den] + \alpha(z-\cn)). \]
The picture can be easily reinforced with a toy example in $d=1$, taking Gaussian wave packets $f(t)=g(t)=e^{-\pi |t|^2}$ antipodally travelling with $\an=(r^{(n)},0)$ and $\bn=(-r^{(n)},0)$ for a sequence $\abs{r^{(n)}} \to \infty$, so that $\cn=c=0$ and
\[ \mathcal{S}_n(x,\xi) = 2 \sqrt{2}e^{-2\pi(|x|^2+|\xi|^2)} \cos( -4\pi r^{(n)}\xi).\]
Measuring the $L^p$-energy of $\mathcal{S}_n$ within an observation window $\Omega$ thus exhibits two effects: As $\cn\to c$ the envelope remains within $\Omega-c$, while intensity fringes separated by $\sim 1/|\den|$ become closer and closer as $\abs{\dn} \to \infty$, in fact leading to equidistribution of linear phases in the asymptotic regime (ultimately by the Riemann-Lebesgue lemma). Intuitively speaking, the separation length goes beyond the resolution power of any possible measuring procedure, which in fact eventually records the average $L^p$-brightness, explaining the occurrence of the visibility factor $\cC_p$. This matches the parallelogram-law intuition when $p=2$ and also the expected $L^\infty$ behavior since $\cC_p \to 1$ as $p \to \infty$. 

The second part, Equation~\eqref{eq-polar2} gives an inequality in terms of $L$, hence requires to control the cross-interference energy by a single signal in light of \eqref{eq-WigSup}. This is done in the proof by introducing the functions $h^\pm_\theta = f \pm e^{i\theta}g$ for every $\theta \in [0,2\pi]$, which should not come as a surprise in view of the previous discussion. Indeed, an easy polarization argument yields 
\[ W(h_\theta^+)-W(h_\theta^-)=4\Re\big(e^{-i\theta} W(f,g)\big) = 4 \abs{W(f,g)} \cos(\alpha-\theta), \]
and the claim follows by measuring the $L^p$-energy (i.e., taking the $L^p$ norm) \textit{after} averaging over $\theta$. This is needed to remove the dependence on the (unknown) local phase $\alpha$, concretely, by replacing $|\cos(\alpha-\theta)|^p$ with its angular mean $\mathcal{C}_p^p$. In addition, it is done to make the right-hand side independent of $\theta$, indeed, $\|h_\theta^+\|_2^2+\|h_\theta^-\|_2^2=2(\|f\|_2^2+\|g\|_2^2)$. 

To summarize these remarks, Theorem~\ref{thm-avgbd} shows that two asymptotically separated packets generally interfere via Wigner distribution to contribute to the $L^p$ norm over a given phase space region by an amount that can be determined (roughly speaking: $\sim$ cross-Wigner envelope mass $\times$ averaged visibility of oscillating carrier), and which cannot exceed the optimal value $L$.
 
\begin{proof}[Proof of Theorem~\ref{thm-avgbd}]
With the notation introduced in the previous discussion, the change of variable \(z'=z-\cn\) yields
		\begin{align}
			\lim_{n\to\infty} & \|\Re\left(W(\pi(\an)f,\pi(\bn)g)\right)\|_{L^p(\Omega)}^p \\
            &= \lim_{n\to \infty} \int_{\Omega} \abs{\Re\left(W(\pi(\an)f,\pi(\bn)g)(z)\right)}^p \dd z \\
			& = \lim_{n\to \infty} \int_{\Omega_n} \abs{\Re\left(e^{i(\exn{\phi} +[z',\den])} W(f,g)(z')\right)}^p \dd z'\\
			& = \lim_{n\to \infty} \int_{\Omega_n} \abs{W(f,g)(z')}^p\abs{\cos([z',\den]+\exn{\phi}+\alpha(z'))}^p \dd z',
		\end{align} where we set $\Omega_n=\Omega-\cn$. Therefore, after setting $F(z)=\abs{W(f,g)(z)}^p$ and $H(\theta)=\abs{\cos(\theta)}^p$, our claim reads 
		\begin{equation}
			\lim_{n\to\infty} \int_{\Omega_n} F(z) H([z,\den]+\exn{\phi}+\alpha(z)) \dd z =\Big(\frac{1}{2\pi}\int_0^{2\pi} H(\theta) \dd \theta \Big)\int_{\Omega-c} F(z) \dd z.
		\end{equation} 
		We argue by approximation with trigonometric polynomials as detailed below. 
		
		\noindent \textbf{Step 1.} We first prove that the claim holds when $H(\theta)$ is replaced by a complex exponential \(P(\theta)=e^{ik\theta}\) with $k \in \Z$. 
		
		For \(k=0\) the problem boils down to showing that \begin{align*}
			\lim_{n\to\infty} \int_{\Omega_n} F(z) \dd z= \int_{\Omega-c}F(z) \dd z,
		\end{align*} which follows at once using continuity on $L^1$ of translates of the indicator function $\chi_\Omega$ of $\Omega$ and the bound in \eqref{eq-wignorm-bds}: 
        \begin{equation}
            \abs{\int_{\rdd} F(z) (\chi_{\Omega}(z+\cn)-\chi_\Omega(z+c)) \dd z} \le \norm{F}_{L^\infty} \norm {\chi_\Omega(z+\cn) - \chi_\Omega(z+c)}_{L^1_z} \to 0.
        \end{equation} 

		For \(k\neq 0\) the right-hand side vanishes, so we must show that 
		\begin{equation}
			\lim_{n\to\infty} \int_{\Omega_n} F(z) e^{ik([z,\den]+\exn{\phi}+\alpha(z))} \dd z = 0.
		\end{equation}
		If we now let \(G_n(z)=F(z) e^{ik\alpha(z)}\chi_{\Omega}(z+\cn)\), the integral on the left-hand side is \begin{align*}
			\int_{\Omega_n} F(z)e^{ik([z,\den]+\exn{\phi}+\alpha(z))} \dd z&=e^{ik \exn{\phi}}\int_{\R^{2d}}G_n(z)e^{ik[z,\den]} \dd z\\
			&=e^{ik \exn{\phi}} \widehat{G_n}(-kJ\dn).
		\end{align*}
		Since \(G_n\in L^1(\R^{2d})\) and \(\abs{\dn} \to \infty\) with \(n\), we aim to invoke the Riemann-Lebesgue lemma to conclude that this integral vanishes in the limit. This is not possible at once, since \(G_n\) also depends on \(n\). However, it is enough to introduce the function \(G(z)=F(z) e^{ik\alpha(z)}\chi_{\Omega}(z+c) \) and arguing as before we have \(\|G_n-G\|_{L^1}\to 0\). As a result, we obtain 
		\begin{align*}
			\abs{\int_{\Omega_n} F(z)e^{ik([z,\den]+\exn{\phi}+\alpha(z))} \dd z} &=\left|e^{ik\exn{\phi}}\widehat{G_n}(-kJ\dn)\right|\\&\leq |\widehat{G_n}(-kJ\dn)-\widehat{G}(-kJ\dn)|+|\widehat{G}(-kJ\dn)|
			\\&\leq \|G_n-G\|_{L^1}+|\widehat{G}(-kJ\dn)| \to 0,
		\end{align*} 
		by virtue of the previous remark and the Riemann-Lebesgue lemma. 
		
		Therefore, by linearity we infer for any trigonometric polynomial \(P(\theta)=\sum_{m=-N}^N c_m e^{im\theta}\) that
		\begin{align*}
			\lim_{n\to\infty} \int_{\Omega_n}F(z) P([z,\den]+\exn{\phi}+\alpha(z)) \dd z= \Big(\frac{1}{2\pi}\int_0^{2\pi}P(\theta) \dd \theta \Big) \int_{\Omega-c} F(z) \dd z.
		\end{align*}
		
		\noindent \textbf{Step 2.} Given an \(\varepsilon>0\) we let \(P_{\varepsilon}(\theta)\) be a trigonometric polynomial such that \(\|H-P_\varepsilon\|_{L^\infty(\T)}<\varepsilon\). For simplicity, given a function \(G\) on the circle we will write \(I_n(G)=\int_{\Omega_n}F(z) G([z,\den]+\exn{\phi}+\alpha(z)) \dd z\) and \(\overline{G}= \frac{1}{2\pi} \int_0^{2\pi} G(\theta) \dd \theta\). We now get \begin{multline}
			\abs{I_n(H)-\overline{H} \norm{F}_{L^1(\Omega-c)}} \leq
			\abs{I_n(H)-I_n(P_\varepsilon)}+\abs{I_n(P_\varepsilon)-\overline{P_\varepsilon} \norm{F}_{L^1(\Omega-c)}}\\ +\norm{F}_{L^1(\Omega-c)}\abs{\overline{P_\varepsilon}-\overline{H}}.
		\end{multline}
		For what concerns the first term we have
		\begin{equation}
			|I_n(H)-I_n(P_\varepsilon)| \leq \norm{H-P_\varepsilon}_{L^\infty} \norm{F}_{L^1(\Omega-\cn)} <\varepsilon \norm{F}_{L^1(\Omega-\cn)}.
		\end{equation} Since \(\norm{F}_{L^1(\Omega-\cn)}\to \norm{F}_{L^1(\Omega-c)}\) (cf.\ Step 1), there exists $C>0$ such that  \(|I_n(H)-I_n(P_\varepsilon)|<C\varepsilon\). Moreover, as proved in Step 1, the second term vanishes in the limit and can therefore be made smaller than \(\varepsilon\) by picking \(n\) large enough. The desired result then follows after bounding the third term with $\abs{\overline{P_\varepsilon}-\overline{H}} \le \varepsilon$.
		
		\bigskip
		
		To conclude, we need to prove the inequality 
\begin{equation*}
	2\mathcal{C}_p\|W(f,g)\|_{L^p(\Omega-c)}\leq L(\|f\|_{L^2}^2+\|g\|_{L^2}^2).
\end{equation*} 
For \(\theta\in [0,2\pi]\) we write \(h_\theta^{\pm}=f\pm e^{i\theta}g\), hence
\begin{align*}
			W(h_{\theta}^{\pm})=Wf+Wg\pm 2 \Re\left(e^{-i\theta}W(f,g)\right), 
		\end{align*}
		and thus \begin{align*}
			W(h_{\theta}^{+})-W(h_{\theta}^{-})=4\Re\left(e^{-i\theta}W(f,g)\right).
		\end{align*}
		After raising to the power \(p\) and integrating in $z$ and $\theta$ we get\begin{align*}
			\frac{1}{2\pi}\int_{0}^{2\pi}\|W(h_{\theta}^{+})-W(h_{\theta}^{-})\|_{L^p(\Omega-c)}^p \dd \theta= \frac{1}{2\pi}\int_{0}^{2\pi}\|4\Re\left(e^{-i\theta}W(f,g)\right)\|_{L^p(\Omega-c)}^p \dd \theta.
		\end{align*}
		Let us now treat each side of this equation separately. If \(\alpha(z)\) denotes the phase of the complex number \(W(f,g)(z)\), so that we can write \(W(f,g)(z)=|W(f,g)(z)|e^{i\alpha(z)}\), the right-hand side equals (by Fubini's theorem)
		\begin{equation}
			\frac{4^p}{2\pi}\int_{\Omega-c}\int_{0}^{2\pi} \abs{\cos(\alpha(z)-\theta)}^p \abs{W(f,g)(z)}^p \dd \theta \dd z =4^p\mathcal{C}_p^p \norm{W(f,g)}_{L^p(\Omega-c)}^p.
		\end{equation}
		On the other hand, to estimate the left-hand side we use Minkowski's inequality to get
		\begin{align*}
			\frac{1}{2\pi}\int_{0}^{2\pi}\|W(h_{\theta}^{+})-W(h_{\theta}^{-})\|_{L^p(\Omega-c)}^p \dd \theta&\leq
			\frac{1}{2\pi}\int_{0}^{2\pi}\left(\|W(h_{\theta}^{+})\|_{L^p(\Omega-c)}+\|W(h_{\theta}^{-})\|_{L^p(\Omega-c)}\right)^p \dd \theta
			\\
			&\leq  \frac{1}{2\pi}\int_{0}^{2\pi}L^p\left(\|h_{\theta}^{+}\|_{L^2}^2+\|h_{\theta}^{-}\|_{L^2}^2\right)^p \dd \theta\\
			&=\frac{1}{2\pi}\int_{0}^{2\pi}2^pL^p\left(\|f\|_{L^2}^2+\|g\|_{L^2}^2\right)^p \dd \theta\\ &=2^pL^p\left(\|f\|_{L^2}^2+\|g\|_{L^2}^2\right)^p,
		\end{align*}
		where we also used the invariance of $L$ under shifts in the Problem \eqref{eq-WigSup}, since by covariance \eqref{eq-covwig} we have
        \begin{equation} \norm{Wf}_{L^p(\Omega-c)} = \norm{Wf(\cdot-c)}_{L^p(\Omega)} = \norm{W(\pi(c)f)}_{L^p(\Omega)} \le L \norm{\pi(c)f}_{L^2}^2 = L \norm{f}_{L^2}^2. 
        \end{equation} The claim thus follows. 
\end{proof}

\subsection{Properties of inequality optimizers} 
	By inspecting the proof of Theorem~\ref{thm-avgbd}, we obtain the following characterization of the inequality's (non-trivial) optimizers.
	
	\begin{corollary} \label{cor-optpol}
	Let $f,g \in L^2(\rd) \setminus\{0\}$. The pair $(f,g)$ achieves equality in the corresponding cross-term inequality from Theorem~\ref{thm-avgbd}, i.e.,
	\[
		2\,\cC_p\,\|W(f,g)\|_{L^p(\Omega)}=L(\|f\|_{L^2}^2+\|g\|_{L^2}^2)
	\]
	if and only if the following conditions hold:
		\begin{enumerate}[label=(\roman*)]
            \item The functions $h_\theta^\pm = f \pm e^{i\theta}g$ are optimizers for the problem \eqref{eq-WigSup} for every $\theta \in [0,2\pi]$. 
			\item For every $\theta \in [0,2\pi]$ we have
			\begin{equation}\label{eq-propA}
				\|W(h_{\theta}^+)-W(h_{\theta}^-)\|_{L^p(\Omega)}=\|W(h_{\theta}^+)\|_{L^p(\Omega)}+\|W(h_{\theta}^-)\|_{L^p(\Omega)}.
			\end{equation}
		\end{enumerate}
\end{corollary}

It is clear that simultaneous occurrence of these conditions place strong constraints on the structure of non-trivial inequality extremizers. Indeed, we illustrate below some of those that can be easily inferred. 

\begin{corollary}\label{cor-Xzero}
        Let $f,g \in L^2(\rd)\setminus\{0\}$ be optimizers for the inequality \eqref{eq-polar2}. Then \eqref{eq-propA} implies $Wf(z)=-Wg(z)$ for almost every \(z\in \Omega\).
\end{corollary}

\begin{proof} Recall that $h_\theta^\pm = f\pm e^{i\theta}g$, $\theta \in [0,2\pi]$. We write \(X=Wf+Wg\) and \(Y_\theta=2\Re(e^{-i\theta}W(f,g))\), so that $W(h_\theta^\pm) = X \pm Y_\theta$, and the claim amounts to showing that $X=0$ a.e.\ on $\Omega$. 

Equality in Minkowski's inequality as in \eqref{eq-propA} for $1\le p<\infty$ implies that, for every $\theta \in [0,2\pi]$, 
\[ \|W(h_{\theta}^+)-W(h_{\theta}^-)\|_{L^p(\Omega)} = \| |W(h_{\theta}^+)|+|W(h_{\theta}^-)|\|_{L^p(\Omega)},\]
from which we infer the pointwise identity
\[ |W(h_{\theta}^+)(z)-W(h_{\theta}^-)(z)|=|W(h_{\theta}^+)(z)|+|W(h_{\theta}^-)(z)| \qquad \text{for a.e. }z \in \Omega.  \]
Let $D \subset [0,2\pi]$ be a countable dense subset and for each $\theta\in D$ let $N_\theta\subset\Omega$ be a null set outside which the above identity holds. We also introduce the full-measure subset $E= \Omega\setminus\bigcup_{\theta\in D}N_\theta$ so that for every $z \in E$ the identity holds for every $\theta \in D$. 

Expanding the Wigner distributions and since $X,Y_\theta$ are real-valued, this amounts to
   \begin{equation}
    2|Y_\theta(z)|=|X(z)+Y_\theta(z)|+|X(z)-Y_\theta(z)| = 2 \max\{ |X(z)|, |Y_\theta(z)|\}\qquad \forall\ z \in E, \, \forall\ \theta \in D,
    \end{equation}
    from which we deduce that \(|Y_\theta(z)|\geq |X(z)|\) for all $z \in E$ and $\theta \in D$. On the other hand, setting $W(f,g)(z)=\abs{W(f,g)(z)}e^{i\alpha(z)}$, by continuity we have
    \begin{equation}
        |X(z)|\le \inf_{\theta \in D}|Y_\theta(z)|=2|W(f,g)(z)| \, \inf_{\theta \in D} \abs{\cos(\alpha(z)-\theta)}=0, \quad \forall z \in E,
    \end{equation} hence we infer $X(z)=0$ for a.e.\ $z \in \Omega$ as claimed. 
   \end{proof}

\begin{corollary}\label{cor-ortho}
Let $f,g \in L^2(\rd)\setminus\{0\}$ be optimizers for the inequality \eqref{eq-polar2}. Then $\langle f,g \rangle =0$. 
\end{corollary}

\begin{proof}
    From Corollary~\ref{cor-Xzero} we know that \(Wf=-Wg\) almost everywhere on \(\Omega\), hence for every $\theta \in [0,2\pi]$ we have
    \begin{align*}
        \|W(f+e^{i\theta}g)\|_{L^p(\Omega)}&=\|Wf+Wg+2\Re(e^{-i\theta}W(f,g))\|_{L^p(\Omega)} \\& =\|2\Re(e^{-i\theta}W(f,g))\|_{L^p(\Omega)}\\
        &=\|-2\Re(e^{-i\theta}W(f,g))\|_{L^p(\Omega)}=\|W(f-e^{i\theta}g)\|_{L^p(\Omega)}.
    \end{align*}
    Furthermore since both \(f+e^{i\theta}g\) and \(f+e^{i(\theta+\pi)}g=f-e^{i\theta} g\) are optimizers, this yields
    \begin{align*}
        L\|f+e^{i\theta}g\|_{L^2}^2=\|W(f+e^{i\theta}g)\|_{L^p(\Omega)}=\|W(f-e^{i\theta}g)\|_{L^p(\Omega)}=L\|f-e^{i\theta}g\|_{L^2}^2.
    \end{align*}
    Since \(L> 0\) we conclude that \( \|f+e^{i\theta}g\|_{L^2}^2=\|f-e^{i\theta}g\|_{L^2}^2\) for every $\theta \in [0,2\pi]$. By expanding, we see that \begin{align*}
        \|f\pm e^{i\theta}g\|_{L^2}^2=\|f\|_{L^2}^2+\|g\|_{L^2}^2\pm 2\Re\langle f, e^{i\theta} g\rangle,
    \end{align*} so it must be that
        \begin{equation}
        \Re (e^{-i\theta}\langle f, g\rangle) = \Re(\langle f,g \rangle)\cos\theta+\Im(\langle f,g \rangle) \sin \theta = 0, \qquad \forall \theta  \in [0,2\pi],
    \end{equation} which clearly forces $\langle f,g\rangle=0$. 
\end{proof}

\subsection{Failure of semicontinuity}
The persistence of the cross-Wigner block $\cS_n$ when the two packets $\pi(a^{(n)})f$ and $\pi(b^{(n)})g$ separate antipodally in phase space while their midpoint remains bounded is a major obstruction to a direct-method proof of the existence of an optimizer. Indeed, this causes both the lack of sequential weak upper semicontinuity, as we now prove more explicitly, and the non relative compactness of a maximizing sequence.

\begin{prop}\label{prop:wuSCfail}
Let $\Omega\subset\R^{2d}$ be measurable with $0<|\Omega|<\infty$, and $1\le p<\infty$. The functional
\[
J(f)\coloneqq \|Wf\|_{L^p(\Omega)},\qquad f\in L^2(\R^d),
\]
is not sequentially weakly upper semicontinuous on $L^2(\R^d)$ at any point. 
\end{prop}

\begin{proof} 
Consider $g(y)=e^{-\pi |y|^2}$, so that $Wg>0$ on $\R^{2d}$. Let $(\xi_n)\subset\R^d$ be a sequence such that $|\xi_n|\to\infty$ and set
\[
a^{(n)}\coloneqq (0,\xi_n),\qquad b^{(n)}\coloneqq (0,-\xi_n), \qquad v_n\coloneqq \pi(a^{(n)})g+\pi(b^{(n)})g.
\]
As a straightforward consequence of the Riemann-Lebesgue lemma we have $v_n\rightharpoonup 0$ in $L^2(\rd)$ and
\[
\|v_n\|_{L^2}^2  =2\|g\|_{L^2}^2+2\Re\int_{\R^d} \abs{g(y)}^2e^{4\pi i \xi_n\cdot y}\,dy \to 2\|g\|_{L^2}^2. 
\]
By sesquilinearity of $W$ we also have 
\[
W(v_n) = W(\pi(a^{(n)})g)+W(\pi(b^{(n)})g)  +\cS_n,
\]
with the symmetric cross-block 
\[
\cS_n \coloneqq W\big(\pi(a^{(n)})g,\pi(b^{(n)})g\big) +W\big(\pi(b^{(n)})g,\pi(a^{(n)})g\big).
\]
We now claim that we have
\[
\|W(\pi(a^{(n)})g)\|_{L^p(\Omega)}\to 0,\qquad \|W(\pi(b^{(n)})g)\|_{L^p(\Omega)}\to 0,
\]
Indeed, if \(\Omega\) is compact then by covariance (Lemma~\ref{lem-covwig}) we have 
\[\|W(\pi(a^{(n)})g)\|_{L^p(\Omega)}=\|Wg\|_{L^p(\Omega-a^{(n)})}\leq |\Omega|^{1/p}\sup_{z\in \Omega-a^{(n)}}|Wg(z)|.\]
Since \(\Omega\) is compact, thus bounded, and \(Wg\) has \(C_0\)-decay, the supremum in the above line vanishes as \(n\to \infty\), since \(|a^{(n)}|\to \infty\), and likewise for \(b^{(n)}\). For a general finite measure set \(\Omega\) we use inner regularity to pick a compact set \(K\subset\Omega\) such that \(|\Omega\setminus K|<\varepsilon\) and combine the argument for compact sets with the triangle inequality.

Now, if we set
\[
\varepsilon_n\coloneqq \|W(\pi(a^{(n)})g)\|_{L^p(\Omega)}+\|W(\pi(b^{(n)})g)\|_{L^p(\Omega)}\to0,
\]
the triangle inequality gives 
\[
\|\cS_n\|_{L^p(\Omega)}-\varepsilon_n
 \le \|W(v_n)\|_{L^p(\Omega)} \le
\|\cS_n\|_{L^p(\Omega)}+\varepsilon_n.
\]
By Theorem~\ref{thm-avgbd} with $d^{(n)}=(0,2\xi_n)$ and $c^{(n)}=0$, we have $\|\cS_n\|_{L^p(\Omega)}\ \longrightarrow\ 2\,\cC_p\,\|Wg\|_{L^p(\Omega)}$, hence
\[
\|W(v_n)\|_{L^p(\Omega)}\ \longrightarrow\ B \coloneqq 2\,\cC_p\,\|Wg\|_{L^p(\Omega)}>0  = \|W(0)\|_{L^p(\Omega)} = J(0),
\]
which proves that $J(f)=\|Wf\|_{L^p(\Omega)}$ is not sequentially weakly upper semicontinuous at $0$.

Let us now consider an arbitrarily chosen $f \in L^2(\rd)$, and for $\lambda>0$ set $g_{n,\lambda}=f+\lambda v_n$, so that $g_{n,\lambda}\rightharpoonup f$. We have
\begin{equation}\label{eq-Wasymptotic}
    W(g_{n,\lambda})=Wf+\lambda^2W(v_n)+2\lambda\Re W(f,v_n).
\end{equation}
Again by covariance, $C_0$-decay and regularity we have
\[
\|W(f,v_n)\|_{L^p(\Omega)}
\le \|W(f,g)\|_{L^p(\Omega-a^{(n)}/2)}+\|W(f,g)\|_{L^p(\Omega-b^{(n)}/2)}\to 0.
\]
By~\eqref{eq-Wasymptotic} $\|W(v_n)\|_{L^p(\Omega)}\to B>0$, hence by the triangle inequality we obtain
\[
\liminf_{n\to\infty}\|W(g_{n,\lambda})\|_{L^p(\Omega)}\ \ge\ \lambda^2 B-\|Wf\|_{L^p(\Omega)}.
\]
Choosing $\lambda$ large enough so that $\lambda^2 B>2\|Wf\|_{L^p(\Omega)}$ yields
$\limsup_{n}\|W(g_{n,\lambda})\|_{L^p(\Omega)}>\|Wf\|_{L^p(\Omega)}$, proving failure of sequential weak upper semicontinuity at $f$.
\end{proof}

\section{Proof of Theorem~\ref{BigTheorem}}\label{sec-proofbig}
First of all, it is clear that since $|\Omega|>0$ we have $L>0$. Moreover, by \eqref{eq-wignorm-bds} we obtain  \(\left(\int_{\Omega}|Wf(z)|^p \dd z\right)^{1/p}\leq 2^d|\Omega|^{1/p}\|f\|_{L^2}^2\), from which we infer that \(L<\infty\). 

\subsection{Profile decomposition}
Let now \(\fn\) be a maximizing sequence for the problem \eqref{eq-WigSup}, and assume without loss of generality that \(\|\fn\|_{L^2}=1\) for all \(n\). We perform profile decomposition as detailed in Lemma~\ref{lem-profdec}, passing to subsequences wherever needed, and write for simplicity \(\Fkn=\sum_{j=1}^k \pi(\zjn)f_j\), so that \(\fn=\Fkn + \wkn \). By sesquilinearity and covariance of the Wigner distribution we have
\begin{align} \label{eq-wigprofdec}
    W(\fn)=&\sum_{j=1}^k W(\pi(\zjn)f_j)+\sum_{1\leq j\neq j' \leq k} W(\pi(\zjn)f_j,\pi(\zjpn)f_{j'})\\
    &+W(\Fkn,\wkn)+W(\wkn,\Fkn)+W(\wkn)\\
    =&\sum_{j=1}^k W(f_j)(\cdot-\zjn)+2\sum_{1\leq j<j' \leq k} \Re \left(W(\pi(\zjn)f_j,\pi(\zjpn)f_{j'})\right)\\
    &+W(\Fkn,\wkn)+W(\wkn,\Fkn)+W(\wkn).
\end{align}
Let us analyze this expression term by term.

\subsection*{Remainder terms}
We first deal with the three remainder terms containing \(\wkn\), and show that the term \(W(\Fkn,\wkn)\) vanishes in the limit by arguing as in \cite{AmbConcentration}. Similar arguments apply to the other remainder terms with obvious modifications. 

We resort to \eqref{eq-WL2Linfty-L2K} and Theorem~\ref{thm-embed} to obtain, for any compact \(K\subset \R^{2d}\),
\begin{equation}
    \|W(\Fkn,\wkn)\|_{L^2(K)} \lesssim_K \|W(\Fkn,\wkn)\|_{W(L^2,L^\infty)}\lesssim \|\Fkn\|_{L^2}\|\wkn\|_{M^\infty}.
\end{equation}
We emphasize that the implied constants are independent of \(n\) and \(k\), but in general may depend on $K$. Thus, since by Lemma~\ref{lem-profdec} we can choose $n$ sufficiently large to ensure $\norm{\Fkn}_{L^2}^2+\norm{\wkn}_{L^2}^2 \le 2$, we conclude that
\begin{align*}
     \|W(\Fkn,\wkn)\|_{L^2(K)}\lesssim_K \|\wkn\|_{M^\infty}.
\end{align*} Combining with the trivial bound 
\begin{equation}\label{eq-unifb-wig}
    \|W(\Fkn,\wkn)\|_{L^\infty} \leq 2^d \|\Fkn\|_{L^2}\|\wkn\|_{L^2} \le 2^{d+1},
\end{equation} by H\"older's inequality we get
\begin{align*}
    \|W(\Fkn,\wkn)\|_{L^p(K)}\lesssim_K \|\wkn\|_{M^\infty}^{\theta_p}, \qquad \theta_p=\begin{cases}
    1 & 1 \le p \le 2 \\ \frac{2}{p} & p\ge 2. 
\end{cases}
\end{align*}
This bound suffices to prove that 
\begin{equation}
    \lim_{k \to \infty} \limsup_{n\to \infty} \norm{W(\Fkn,\wkn)}_{L^p(\Omega)} = 0. 
\end{equation} Indeed, if $\Omega$ is compact the conclusion is immediate from the previous argument. If not, it is enough to exhaust by a compact $K \subset \Omega$ with $\abs{\Omega \setminus K}$ arbitrarily small and resort to \eqref{eq-unifb-wig}. 

\subsection*{Structure of profile decomposition}
Let us now examine the main terms appearing in the profile decomposition \eqref{eq-wigprofdec}. 

\subsubsection*{A single compact profile} Since \(|\zjn-\zjpn|\to \infty\) for all \(j\neq {j'}\) by construction, it is easy to argue by contradiction that \(|\zjn|\to \infty\) for all but at most one \(j\), which we denote by \(j_0\) in the case it does exist. Since \(Wf\in C_0(\R^{2d})\) for all \(f\in L^2(\R^d)\), we have, by covariance and inner regularity, \(\|W(f_j)(\cdot - \zjn)\|_{L^p(\Omega)}\to 0\) as \(n\to \infty\) for every \(j\neq j_0\). So, there is at most one potentially surviving pure profile $f_{j_0}\in L^2(\R^d)\setminus\{0\}$, from now on referred to as the \textit{compact profile}, corresponding (up to subsequences) to the bounded, converging sequence $\exn{z}_{j_0} \to z_{j_0} \in \rdd$.  We conveniently set \(f_{j_0}=0\) and \(z_{j_0}^{(n)}\equiv 0\) in the case where the index $j_0$ with the aforementioned properties does not exist.

\subsubsection*{Cross terms and their interactions}
Note that the Wigner covariance relation implies 
\begin{align*}
    \left|W(\pi(\zjn)f_j,\pi(\zjpn)f_{j'})(z)\right|=\left|W(f_j,f_{j'})\left(z-\frac{\zjn+\zjpn}{2}\right)\right|.
\end{align*}
Just like for the pure terms, only the profile pairs for which \(\zjn+\zjpn\) remains bounded as \(n\to \infty\) avoid vanishing in the limit. We will refer to those as \textit{surviving pairs}. 

In fact, if \(\zjn+\zjpn\) remains bounded, then neither \(j\) nor \(j'\) can associate with the compact profile. Indeed, neither \(\zjn\) nor \(\zjpn\) can remain bounded with \(n\), since in that case \(|\zjn-\zjpn|\leq |\zjn+\zjpn|+2\min\{|\zjn|,|\zjpn|\}\) would remain bounded as well. Similarly, neither \(j\) nor \(j'\) can appear in a different surviving pair: for instance, if we also had that \(\zjn+\zln\) remained bounded, then we would have \(|\zjpn-\zln|\leq |\zjpn+\zjn|+|\zjn+\zln|<\infty\) in the limit, another contradiction. 

It is thus natural to introduce the set $\cJ$ consisting of pairs of indices $(j,j')$ with $j<j'$ in the profile decomposition of $\fn$ such that $|\zjn+\zjpn|$ remains bounded, and therefore converges up to subsequence.

\subsection{The main argument}
The previous discussion shows that when we take the limit in \(n\) in \eqref{eq-wigprofdec} the only profiles that eventually contribute are the compact profile \(f_{j_0}\), if it exists, and surviving pairs \((f_j,f_{j'})\) with $(j,j') \in \cJ$. Since profile decomposition holds for any \(k\) we can take the limit, first in \(n\), then in \(k\). By combining our observations about the surviving terms in the expansion of the Wigner distribution, after passing to a common subsequence over $\cJ$ via a diagonal argument we obtain
	\begin{equation}
		L \le \lim_{n\to\infty} \norm{W(\pi(z^{(n)}_{j_0})f_{j_0})}_{L^p(\Omega)} + \lim_{n \to \infty} \sum_{{(j,j')\in \cJ}}  \norm{\cS_{j,j'}^{(n)}}_{L^p(\Omega)},
	\end{equation} where we set $\cS_{j,j'}^{(n)}=2\Re \left(W(\pi(\zjn)f_j,\pi(\zjpn)f_{j'})\right)$ (cf.\ notation in Theorem~\ref{thm-avgbd}). It is clear that 
	\begin{equation} \lim_{n\to\infty}\|W(\pi(z_{j_0}^{(n)})f_{j_0})\|_{L^p(\Omega)}\leq \lim_{n\to\infty}L\|\pi(z_{j_0}^{(n)})f_{j_0}\|_{L^2}^2=L\|f_{j_0}\|_{L^2}^2,
	\end{equation} while invoking Theorem~\ref{thm-avgbd} yields
	\begin{equation}
		\lim_{n \to \infty} \norm{\cS_{j,j'}^{(n)}}_{L^p(\Omega)} = 2\cC_p \norm{W(f_j,f_{j'})}_{L^p(\Omega-c_{j,j'})} \le L(\norm{f_j}_{L^2}^2+\norm{f_{j'}}_{L^2}^2), 
	\end{equation} where $c_{j,j'} \in \rdd$ is the limit of (a subsequence of) $c^{(n)}_{j,j'}$, which is bounded since $(j,j')\in \cJ$. To sum up, we have obtained 
	\begin{align}
		L &  \le \lim_{n\to\infty} \norm{W(\pi(z^{(n)}_{j_0})f_{j_0})}_{L^p(\Omega)}+\sum_{{(j,j')\in \cJ}} \lim_{n \to \infty} \norm{\cS^{(n)}_{j,j'}}_{L^p(\Omega)} \\
		& \le  L\Big(\|f_{j_0}\|_{L^2}^2+\sum_{{(j,j')\in \cJ}}\|f_j\|_{L^2}^2+\|f_{j'}\|_{L^2}^2\Big) \\ 
		& \leq L\sum_{j=1}^\infty \|f_j\|_{L^2}^2 \\
		&\le L.
	\end{align} Note that in the second inequality we crucially used the matching structure of the surviving pairs, which avoids double counting. It is thus clear that every inequality in this chain must be in fact an equality. In particular, we get 
	\begin{equation}
		\|f_{j_0}\|_{L^2}^2+\sum_{{(j,j') \in \cJ}} \Big(\|f_j\|_{L^2}^2+\|f_{j'}\|_{L^2}^2\Big) = 1. 
	\end{equation}
	Now if \(f_{j_0}\neq 0\) then \(\lim_{n\to \infty} \|W(\pi(z^{(n)}_{j_0})f_{j_0})\|_{L^p(\Omega)}=L\|f_{j_0}\|_{L^2}^2\). Since \(z_{j_0}^{(n)}\) converges (up to subsequence) to \( z_0 \in \rdd\) we get, by employing the same argument as in the proof of Theorem~\ref{thm-avgbd},
	\begin{equation}
		\lim_{n\to \infty} \|W(\pi(z_{j_0}^{(n)})f_{j_0})\|_{L^p(\Omega)}=\|Wf_{j_0}\|_{L^p(\Omega-z_0)}=\|W(\pi(z_0)f_{j_0})\|_{L^p(\Omega)},
	\end{equation}
	hence showing that \(\pi(z_0)f_{j_0}\) is an optimizer. 
	
	On the other hand, if no compact profile exists we must examine the cross terms. Since every inequality is an equality in the above chain, for each pair $(j,j') \in \cJ$ we must have
    \begin{align*}
		2\mathcal{C}_p\|W(f_{j},f_{j'})\|_{L^p(\Omega-c_{j,j'})}=L(\|f_j\|_{L^2}^2+\|f_{j'}\|_{L^2}^2).
	\end{align*}
    If one of $f_j$ or $f_{j'}$ is zero, then they must both be zero. Otherwise, in view of
	Corollary~\ref{cor-optpol}, after setting \(h_{j,\theta}^{\pm}=f_j\pm e^{i\theta}f_{j'}\) for $\theta \in [0,2\pi]$,
    by covariance we conclude that
	\begin{align*}
		L\|h_{j,\theta}^{\pm}\|_{L^2}^2=\|W(h_{j,\theta}^{\pm})\|_{L^p(\Omega-c_{j,j'})}=\|W(\pi(c_{j,j'})h_{j,\theta}^{\pm})\|_{L^p(\Omega)}.
	\end{align*}
    Observe that by Corollary~\ref{cor-ortho} we have
    \(\|h_{j,\theta}^{\pm}\|_{L^2}^2=\|f_j\|_{L^2}^2+\|f_{j'}\|_{L^2}^2\neq 0\).
    In particular, this means that there exists a one-parameter family of optimizers for the problem \eqref{eq-WigSup}, namely $\{ \pi(c_{j,j'})(f_j + e^{i\theta} f_{j'}) : \theta \in [0,2\pi]\}$. 
	
	To summarize, if \(\|f_{j_0}\|_{L^2}^2\neq 0\), then \(\pi(z_0)f_{j_0}\) is an optimizer. If no compact profile exists then there is at least one surviving pair \((f_j,f_{j'})\) of non-trivial profiles which generates a family \(\{\pi(c_{j,j'})(f_j + e^{i\theta}f_{j'})\}_{\theta\in [0,2\pi]}\) of optimizers. In any case, we can conclude that the supremum \eqref{eq-WigSup} is attained. \qed 

\subsection{Comments on exotic optimizers}\label{sec:Exotic}
Theorem~\ref{BigTheorem} establishes the existence of an optimizer for the Problem \eqref{eq-WigSup}. However, the proof reveals two very different situations: up to shifts, the optimizer may be the single compact profile extracted from a profile decomposition or any of those belonging to the one-parameter family associated with a surviving pair. 

It must be stressed that, in the second (degenerate) case, the surviving pair \((f_j,f_{j'})\) happens to be an optimizer of the inequality \eqref{eq-polar2}, hence it comes along with the additional constraints highlighted in Corollary~\ref{cor-optpol}, some of which have been deduced in Corollaries~\ref{cor-Xzero} and~\ref{cor-ortho}. In view of these results, it is clear that if $(f_j,f_{j'})$ is a surviving pair in the profile decomposition \eqref{eq-wigprofdec}, then it comes along with a family of optimizers $\pi(c_{j,j'})(f_j +e^{i\theta}f_{j'})$ with $\theta \in [0,2\pi]$ satisfying the following constraints:
\begin{equation}
    Wf_j(z)=-Wf_{j'}(z) \quad \text{for a.e. }z \in \Omega-c_{j,j'}, \qquad \langle f_j,f_{j'}\rangle=0.
\end{equation}
The optimizing contribution on $\Omega$ is then purely due to constructive interferences, since
\begin{equation}
    \norm{W(\pi(c_{j,j'})(f_j +e^{i\theta}f_{j'}))}_{L^p(\Omega)} = \norm{2\Re (e^{-i\theta} W(f_j,f_{j'}))}_{L^p(\Omega-c_{j,j'})}.
\end{equation}
While these conditions seem not to be a priori incompatible, it is evident that they place extremely rigid constraints on the exotic optimizers. Still, they are not strong enough to completely rule out the occurrence of such optimizers and have so far eluded our attempts to exhibit examples of fitting functions. Nevertheless, we can claim that the occurrence of exotic optimizers is rare in the grand scheme. More precisely, the subset of pairs of exotic optimizers is meagre in $(L^2(\rd))^2$, being a subset of the set $\mathcal{N}$ of optimizers for the inequality \eqref{eq-polar2}, which is in turn a subset of the set $\mathcal{O}$ of pairs of orthogonal functions in $(L^2(\rd))^2$ as a consequence of Corollary~\ref{cor-ortho}. It is straightforward to prove that the latter is a nowhere dense subset, being indeed closed with empty interior. In particular, for every $(f,g) \in \mathcal{O}$ one can find a pair of non-zero directions $h,k \in L^2(\rd)$ such that the family $(f+th,g+tk)$, $t>0$, belongs to $\mathcal{O}^c$ while converging to $(f,g)$. 

\begin{remark} In fact, by refining the previous argument it is possible to show that $\mathcal{O}$, hence the exceptional set \(\mathcal{N}\), satisfies pretty much any other meaningful notion of negligibility, such as being directionally porous, which in turn implies being shy (or equivalently Haar null) and Gaussian null. We avoid such a technical discussion and address the interested reader to \cite{Zajicek} for further details. 
\end{remark}

We are thus left with some open issues. Theorem~\ref{BigTheorem} establishes the existence of an optimizer, but examining the proof reveals several candidates for the optimizer. As such, it is natural to ask the following questions related to uniqueness of the optimizer.

\begin{question} Are we guaranteed that a compact profile \(f_{j_0}\) exists? Are the properties provided in Section~\ref{sec:Exotic} actually incompatible? \end{question}

\section{The $L^\infty$ concentration problem} \label{sec-linfty}
In the case \(p=\infty\), the existence of an optimizer is significantly easier to establish.
\begin{prop}
    Let \(\Omega\subset \R^{2d}\) be measurable such that \(0<|\Omega|<\infty\). We have
    \begin{equation}\label{InfinityProblem}
        \sup_{f\in L^2(\R^d)\setminus\{0\}}\frac{\|Wf\|_{L^\infty(\Omega)}}{\|f\|_{L^2}^2}=2^d,
    \end{equation}
    and the supremum is attained.
\end{prop}
\begin{proof}
    Recall from \eqref{eq-wignorm-bds} the $L^\infty$ estimate
    \begin{align*}
        \|Wg\|_{L^\infty(\Omega)}\leq \|Wg\|_{L^\infty}\leq 2^d \|g\|_{L^2}^2.
    \end{align*}
    It is clear that $\|Wg\|_{L^\infty}$ is attained at some $w_0 \in \rdd$, since $Wg \in C_0(\rdd)$ for $g \in L^2(\rd) \setminus\{0\}$. Pick a point of positive Lebesgue density $c \in \Omega$ (recall that $\Omega$ has positive measure) and set $f=\pi(c-w_0)g$. Therefore, $\| W f\|_{L^\infty(\Omega)} \ge |W f(c)| = |W g(w_0)| = \|W g\|_{L^\infty}$, and since by covariance, $|W g(w_0)|=|W(\pi(-w_0)g)(0)|$, we have
\[ \sup_{f\in L^2(\R^d)\setminus\{0\}}\frac{\|W f \|_{L^\infty(\Omega)}}{\|f\|_{L^2}^2} = \sup_{g\in L^2(\R^d)\setminus\{0\}}\frac{\|W g \|_{L^\infty}}{\|g\|_{L^2}^2} = \sup_{g\in L^2(\R^d)\setminus\{0\}}\frac{|W g(0)|}{\|g\|_{L^2}^2}.
\]
If we let \(g\) be an even function, then the claim follows, since
    \begin{align*}
        Wg(0)
        &=\int_{\R^{d}} g(y/2)\overline{g(-y/2)} \dd y\\
        &=\int_{\R^{d}} |g(y/2)|^2 \dd y\\
        &=2^d\int_{\R^{d}} |g(u)|^2 \dd u=2^d\|g\|_{L^2}^2. \qedhere
    \end{align*}
\end{proof}
\begin{remark}
    Notice that the same argument works also when $g$ is odd, since $Wg(0)=-2^d \norm{g}_{L^2}^2$, hence the set of optimizers for \eqref{InfinityProblem} is quite large. On the other hand, there exist non-zero functions whose concentration functional vanishes. For instance, assuming that $\Omega$ is bounded with $R\coloneqq \operatorname{diam}(\Omega)$, if $f$ is supported on $[R+1,R+2]$ then the support property of the Wigner distribution in $d=1$ \cite[Lemma 4.3.5]{Grochenig} implies that $Wf(x,\xi)=0$ whenever $|x|\le R$. Therefore, for any $c=(c_x,c_\xi)\in\Omega$ and any $z=(x,\xi)\in\Omega$ we have $|x-c_x|\le |z-c|\le R$, hence $W(\pi(c)f)(z)=Wf(z-c)=0$ for all $z\in\Omega$. As a result, $\|W(\pi(c)f)\|_{L^\infty(\Omega)}=0$.
    \end{remark}
    
\begin{remark}
    Compared to the case \(p\in [1,\infty)\), where the conclusion is the same for the Wigner distribution as for the ambiguity function, the \(p=\infty\) case is strikingly different. For the ambiguity function, the supremum is only attained if \(0\in \Omega\), but if this is the case, every function will be an optimizer~\cite[Proposition 1.2]{AmbConcentration}. On the other hand, the supremum is always attained for the Wigner distribution, but for a function to be an optimizer it has to satisfy very specific symmetry conditions. This further illustrates the difference between the time-frequency covariant Wigner concentration and the time-frequency invariant ambiguity concentration problems.
\end{remark}

\section{Generalizations and open problems} \label{sec-gen}
In this section we discuss some variations on the main problem arising from considering different yet related time-frequency distributions, namely the \(\tau\)-Wigner distributions and the Born--Jordan distribution. As in the proof of Theorem~\ref{BigTheorem}, we approach the corresponding problems via profile decomposition.  We carry the strategy as far as possible, but certain crucial steps differ from the Wigner case and call for suitable adjustments, leading to very different outcomes: Indeed, we cannot conclude the existence of an optimizer (which remains an open problem for the time being) because the structure of maximizing profiles is significantly more complex (see Lemma~\ref{lem-deg-acycl}) when $\tau \ne 1/2$, whereas for the Born--Jordan problem (roughly speaking, after averaging over $\tau$) we are even able to unlock the direct method by proving full weak continuity. 

\subsection{Concentration problems for $\tau$-Wigner distributions}
Recall that the cross-\textit{\(\tau\)-Wigner distributions} \cite{BDDO,CT_20} are defined for \(f,g\in L^2(\R^d),\; \tau\in (0,1) \) by \begin{align*}
    W_{\tau}(f,g)(z)=\int_{\R^d}e^{-2\pi i \xi \cdot y} f(x+\tau y)\overline{g(x-(1-\tau)y)} \; \dd y.
\end{align*} Letting \(\tau=1/2\) recovers the Wigner distribution. Because of the similar structure, the \(\tau\)-Wigner distributions inherit many of the important properties of the Wigner distribution, for instance covariance under time-frequency shifts: this reads (cf.\ \cite[Proposition 3.1]{CNT_19})
	\begin{equation} \label{eq-covtau}
			W_\tau(\pi(a)f,\pi(b)g)(z)
			= e^{2\pi i\,[z,a-b]}	e^{2\pi i (x_a-x_b)\cdot(\tau \xi_a+(1-\tau)\xi_b)} W_\tau(f,g)\!\big(z-c_{\tau}(a,b)\big),	\end{equation}
where we set \begin{equation} 
	c_{\tau}(a,b)
	\coloneqq \big((1-\tau)x_a+\tau x_b\,,\, \tau \xi_a+(1-\tau)\xi_b\big),
	\qquad a=(x_a,\xi_a),\ b=(x_b,\xi_b).
	\end{equation}
Moreover, Theorem~\ref{thm-embed} also extends to \(W_{\tau}\).

\begin{theorem}[\!\!{\cite[Theorem 4.12]{CT_20}}]\label{thm-embedFull}
    Let \(f\in L^2(\rd)\), \(g\in M^{\infty}(\rd)\). Then \(W_\tau (f,g)\in W(L^2,L^\infty)(\rdd)\) for every $\tau \in (0,1)$, with
		\begin{align*}
			\|W_\tau (f,g)\|_{W(L^2,L^\infty)}\lesssim_d C(\tau) \|f\|_{L^2}\|g\|_{M^{\infty}},
		\end{align*} where \(\displaystyle C(\tau)=\frac{2^d}{(\tau(1-\tau))^{d/2}}\).
\end{theorem}
It is therefore natural to ask whether the conclusion of Theorem~\ref{BigTheorem} holds for the other \(\tau\)-Wigner distributions as well. By the elementary \(L^\infty\)-bound 
\[ \|W_\tau(f,g)\|_\infty\leq {(\tau(1-\tau))^{-d/2}}\|f\|_{L^2}\|g\|_{L^2}\] it is clear that the supremum 
\[ L_\tau=\sup_{f\in L^2(\R^d)\setminus\{0\}}\frac{\left(\int_{\Omega}|W_\tau f(z)|^p \dd z\right)^{1/p}}{\|f\|_{L^2}^2} \]
is finite (and of course positive). In addition, the same argument as in the proof of Proposition~\ref{prop:wuSCfail} shows that the corresponding concentration functional is not weakly upper semicontinuous at any point of \(L^2(\rd)\). We therefore again consider profile decomposition as in Lemma~\ref{lem-profdec}, leading to a similar expansion as in \eqref{eq-wigprofdec}, albeit with $W_\tau$ in place of $W$. In light of Theorem~\ref{thm-embedFull} the error terms can be controlled using exactly the same argument, and the covariance property of \(W_{\tau}\) shows that all pure terms vanish, possibly except for the compact profile \(f_{j_0}\). 

Due to the asymmetry of \(W_\tau\), though, the cross terms exhibit a less controlled behavior. Indeed, the expansion of $W_\tau(f^{(n)})$ in \eqref{eq-wigprofdec} naturally leads to grouping cross terms into symmetrized blocks: for each pair of distinct profile indices $(j,j')$, we set
	\[
	\cS_{j,j'}^{(n)}(z) \coloneqq W_\tau(\pi(\zjn)f_j,\pi(\zjpn)f_{j'})(z)+W_\tau(\pi(\zjpn)f_{j'},\pi(\zjn)f_{j})(z).
	\]
	The asymptotic behavior of this block depends critically on the parameter $\tau$. By virtue of covariance, setting $\zjn=(x_j^{(n)},\xi_{j}^{(n)})$ and $\zjpn=(x_{j'}^{(n)},\xi_{j'}^{(n)})$ we have:
\begin{itemize}
	\item $W_\tau(\pi(\zjn)f_j,\pi(\zjpn)f_{j'})$ is centered at
	\[
	\cnj \coloneqq c_\tau(\zjn,\zjpn)
	=\big((1-\tau)\,x_j^{(n)}+\tau\,x_{j'}^{(n)},\;\;\tau\,\xi_j^{(n)}+(1-\tau)\,\xi_{j'}^{(n)}\big).
	\]
	\item $W_\tau(\pi(\zjpn)f_{j'},\pi(\zjn)f_j)$ is centered at
	\[
	\cnjp \coloneqq c_\tau(\zjpn,\zjn)
	=\big((1-\tau)\,x_{j'}^{(n)}+\tau\,x_j^{(n)},\;\;\tau\,\xi_{j'}^{(n)}+(1-\tau)\,\xi_j^{(n)}\big).
	\]
\end{itemize}
Hence
\[
\cnj-\cnjp
=(1-2\tau)\,\big(x_j^{(n)}-x_{j'}^{(n)},\;-(\xi_j^{(n)}-\xi_{j'}^{(n)})\big). 
\]
Setting $\dnj\coloneqq\zjn-\zjpn$, we thus have
\[
\big|c_\tau(\zjn,\zjpn)-c_\tau(\zjpn,\zjn)\big|
=|1-2\tau|\,|\dnj|, 
\]
and since $|\dnj| \to \infty$ by construction we face a significant dichotomy: 
\begin{itemize}
	\item If $\tau=1/2$ the centers coincide and the block coalesces into an oscillatory function as before: 
	\begin{equation}
		\cSjn=2\Re(W(\pi(\zjn)f_j,\pi(\zjpn)f_{j'})).
	\end{equation}
	\item If $\tau\ne 1/2$ the centers separate along $|\dnj|$ and the two contributions decouple in phase space. 
\end{itemize}
In any case, for the block $\cS_{j,j'}^{(n)}$ to have a non-vanishing contribution to the $L^p(\Omega)$ norm in the limit, at least one of its two centers, $\cnj$ or $\cnjp$, must remain bounded as $n\to\infty$. Just like for the pure terms, only those blocks $\cSjn$ for which at least one of the two centers
$\cnj,\cnjp$ remains bounded as $n\to\infty$ can contribute in the limit, and in that case we refer to $\cSjn$ as a \textit{surviving block}. More precisely, we say the ordered pair $(j,k)$ is \emph{surviving} if the center sequence $\{c^{(n)}_{j,k}\}_n$ is bounded in $\rdd$, while a \emph{chain} is a (finite or infinite) sequence $(j_1,j_2,\dots)$ of distinct indices such that each consecutive ordered pair $(j_r,j_{r+1})$ is surviving. We stress that when $\tau \ne 1/2$ the centers $\cnj$ and $\cnjp$ diverge from one another, hence for any surviving block exactly one of these two center sequences can be bounded, while for $\tau=1/2$ the two centers coincide.

These findings show that the structure of the graph of surviving indices is richer when $\tau \ne 1/2$. Indeed, simple generalizations of the arguments in the case $\tau = 1/2$ provide the following conclusions. 

\begin{lemma} \label{lem-deg-acycl}
	Fix $\tau\in(0,1)$ and consider the directed graph on profile indices where a directed edge $(j,k)$ is placed for each and only surviving pair. Then:
	\begin{enumerate}
		\item A member of a surviving pair cannot be a compact profile: if $(j,k)$ is surviving then $z^{(n)}_j$ and $z^{(n)}_k$ are unbounded sequences.
		\item Each node has in-degree and out-degree $\le 1$: for each $j$ there is at most one $k$ with $(j,k)$ surviving, and at most one $\ell$ with $(\ell,j)$ surviving. In particular, when $\tau= 1/2$ an index can belong to at most one surviving pair.
		\item The graph is acyclic if $\tau\ne 1/2$. 
	\end{enumerate}
	Therefore, the surviving-edge graph is a disjoint union of chains, that for $\tau=1/2$ reduces to a disjoint union of undirected edges (a matching).
\end{lemma}

These findings illustrate that when $\tau \ne 1/2$ and chains do appear in a profile decomposition, the main obstruction in the main chaining argument $L_\tau \le \cdots \le L_\tau$, is avoiding energy double-counting. This issue can be ultimately fixed by combining the definition of $L_\tau$ with the asymptotic orthogonality of the extracted profiles: putting technicalities aside, if $\Pi=(j_1,j_2,\dots)$ is a (finite or infinite) chain from the profile decomposition then one has 
\[
	\limsup_{n\to\infty}\ \Big\| W_\tau \Big( \sum_{j\in\Pi}\pi(z^{(n)}_j)f_j \Big) \Big\|_{L^p(\Omega)}
	\ \le\ L_\tau \sum_{j\in\Pi}\|f_j\|_{L^2}^2. 
	\]
Therefore, the same argument used for the Wigner distribution shows that if the compact profile \(f_{j_0}\neq 0\), then an optimizer exists, but when \(f_{j_0}=0\) we cannot just claim that a surviving block is an asymptotic optimizer. Indeed, the Wigner proof crucially relied on Theorem~\ref{thm-avgbd}, where the limit of the block contribution was explicitly computed and bounded. We do not have a sufficiently strong parallel result for chains, leaving open the existence of optimizers under this strategy. 

\begin{question} Is there an analogue of Theorem~\ref{thm-avgbd} for chains? Can we somehow guarantee that a compact profile exists? \end{question}

On the other hand, we are able to obtain a full answer for the $L^\infty$ optimization problem, marking once again the exceptional nature of the case $\tau=1/2$. 
\begin{prop} \label{prop:tauinfty}
            Let \(\Omega\subset \R^{2d}\) be measurable such that \(0<|\Omega|<\infty\), and \(\tau\in (0,1) \setminus\{1/2\}\). We have
    \begin{align}
        \sup_{f\in L^2(\R^d)\setminus\{0\}}\frac{\|W_\tau f \|_{L^\infty(\Omega)}}{\|f\|_{L^2}^2}=\left(\frac{1}{\tau(1-\tau)}\right)^{d/2},
    \end{align}
    but the supremum is not attained. 
    \end{prop}

\begin{proof}
Arguing as in the Wigner case, the problem boils down to computing
\[ \sup_{g\in L^2(\R^d)\setminus\{0\}}\frac{|W_\tau g(0)|}{\|g\|_{L^2}^2}.
\]
Since
\[
W_\tau g(0)= \int_{\rd} g(\tau y) \overline{g(-(1-\tau)y)} \dd y = (\tau(1-\tau))^{-d/2}\langle D_\tau g,\, D_{-(1-\tau)} g\rangle,
\]
where $D_a g(y)=|a|^{d/2}g(ay)$ is the unitary dilation by $a \ne 0$, we infer
\[
\sup_{g\in L^2(\R^d)\setminus\{0\}}\frac{|\langle D_\tau g, D_{-(1-\tau)}g\rangle|}{\|g\|_{L^2}^2}
  = \sup_{\|g\|_{L^2}=1} |\langle D_{-\frac{\tau}{1-\tau}}g, g\rangle|  
=\|D_{-\frac{\tau}{1-\tau}}\|_{L^2\to L^2} = 1.
\]
We have invoked~\cite[Theorem 12.25]{Rudin} in the second equality. On the other hand, if the supremum were attained it would force \(D_{-\frac{\tau}{1-\tau}}g=\lambda g\) a.e. for some $\lambda \in \C$ with $\abs{\lambda}=1$, which is impossible for \(g\in L^2(\rd)\setminus\{0\}\) unless $\abs{\frac{\tau}{1-\tau}}=1$, that is $\tau=1/2$.
\end{proof}
    
\subsection{The Born--Jordan concentration problem}\label{ssec:BornJordan}
Recall that for $f,g\in L^2(\rd)$ the Born--Jordan distribution is defined as the $\tau$-average of the $\tau$-Wigner distributions:
\[
\BJ(f,g)\,=\,\int_0^1 W_\tau(f,g)\, \dd \tau.
\]
Equivalently, it can be viewed as the Cohen class distribution with kernel
\begin{equation}\label{eq:bj-def}
	\BJ(f,g) =  W(f,g)* \Theta_\sigma, 
	\qquad 
	\Theta(\zeta)=\frac{\sin(\pi\,\zeta_1\!\cdot\!\zeta_2)}{\pi\,\zeta_1\!\cdot\!\zeta_2}.
\end{equation}
While the Born--Jordan distribution comes with its own set of problems stemming from the bad regularity/decay properties of \(\Theta\), it is intuitively clear that symmetry over $\tau$ should lead to a better behavior when it comes to phase space concentration compared to the \(\tau\)-Wigner distributions. In fact, this is consistent with the findings on the reduction of cross- and self-interference fringes compared to the Wigner distribution, cf.\ again \cite{CGT_18}. Furthermore, using the profile decomposition approach we are able to put these heuristics on a rigorous ground by showing a regularity gain of the Born--Jordan distribution's concentration functional compared to the Wigner distribution's. In particular, we show the following continuity result for \(d=1\).
\begin{prop} \label{prop-BJ-wusc}
Let $d=1$, $1\le p<\infty$, and $\Omega\subset\R^2$ be measurable with $0<|\Omega|<\infty$. The functional $J_{\mathrm{BJ}}(f)\coloneqq\|\BJ f\|_{L^p(\Omega)}$ is sequentially weakly continuous in $L^2(\R)$.
\end{prop}

Although it is well known that the averaging process giving rise to the Born--Jordan distribution smooths in the sense that cross- and self-interference fringes are reduced, Proposition~\ref{prop-BJ-wusc} shows that this averaging process even improves regularity on a topological level. This observation could potentially have much broader implications in time-frequency analysis, which we plan to explore in future work, but as far as the scope of this paper goes, Proposition~\ref{prop-BJ-wusc} shows that when \(d=1\) the Born--Jordan concentration problem has a positive answer. 
\begin{theorem} \label{thm:BJ-exist}
	Let $\Omega\subset \R^2$ be measurable with $0<|\Omega|<\infty$ and $1\le p< \infty$. Then
	\[
	L_{\mathrm{BJ}}
	=\sup_{f\in L^2({\R})\setminus\{0\}}\frac{\left(\int_{\Omega}|\BJ f(z)|^p \dd z\right)^{1/p}}{\|f\|_{L^2}^2}
    \]
	is attained by an optimizer.
\end{theorem}
Let us note right away that by Cauchy-Schwarz's inequality we have 
\begin{equation}  \|\BJ f\|_{L^\infty(\R^2)}\leq\int_0^1\|W_{\tau}f\|_{L^\infty(\R^2)}\; \dd \tau\leq \Big(\int_{0}^1\frac{1}{\sqrt{\tau(1-\tau)}}\; \dd \tau\Big)\|f\|_{L^2(\R)}^2=\pi\|f\|_{L^2(\R)}^2,\end{equation} so \(L_{\mathrm{BJ}}\) is finite (and of course positive). Thus, by the direct method, Theorem~\ref{thm:BJ-exist} follows in a straightforward way once Proposition~\ref{prop-BJ-wusc} is proved. To this aim, we will rely on several properties of the Born--Jordan distribution that are now outlined. We start with the covariance formula, which follows by straightforward integration of the corresponding one for $\tau$-Wigner distributions. 

\begin{prop} \label{prop:BJcov}
	For all $a=(x_a,\xi_a),b=(x_b,\xi_b),z \in\R^2$ and $f,g\in L^2(\R)$,
	\begin{align}
		\BJ(\pi(a)f,\pi(b)g)(z)
		&= \int_0^1 W_\tau(\pi(a)f,\pi(b)g)(z)\, \dd \tau \nonumber\\
		&= e^{2\pi i [z,a-b]}\!
		\int_0^1 e^{2\pi i (x_a-x_b)\cdot(\tau\xi_a+(1-\tau)\xi_b)}\,
		W_\tau(f,g)\big(z-c_\tau(a,b)\big)\,\dd \tau, \label{eq:BJcov}
	\end{align}
        where \(c_\tau(a,b)=((1-\tau)x_a+\tau x_b,\tau \xi_a+(1-\tau)\xi_b).\)
	In particular, for joint shifts ($a=b$) we have the usual translation covariance
	\begin{equation}\label{eq:BJjoint}
		\BJ(\pi(a)f,\pi(a)g)(z)=\BJ(f,g)(z-a),
		\qquad \BJ(\pi(a)f)(z)=\BJ f(z-a).
	\end{equation}
\end{prop}

Let us now substantiate the reduction of interferences by proving that cross-terms do not contribute to the concentration problem. 

\begin{lemma} \label{lem:BJ-vanish}
	Let $\Omega\subset\R^2$ have finite measure and $1 \le p < \infty $. Fix $f,g\in L^2(\R)$ and sequences $a^{(n)},b^{(n)}\in\R^2$ with $|a^{(n)}-b^{(n)}|\to\infty$. Then
	\[
	\lim_{n\to\infty}\ \|\BJ(\pi(a^{(n)})f,\pi(b^{(n)})g)\|_{L^p(\Omega)}=0,
	\qquad
	\lim_{n\to\infty}\ \|\BJ(\pi(b^{(n)})g,\pi(a^{(n)})f)\|_{L^p(\Omega)}=0.
	\]
\end{lemma}

\begin{proof}
	Let $\mathcal{B}_n(\tau;\cdot)\coloneqq W_\tau(\pi(a^{(n)})f,\pi(b^{(n)})g)(\cdot)$. By \eqref{eq-covtau},
	\[
	\|\mathcal{B}_n(\tau;\cdot)\|_{L^p(\Omega)}
	=\|W_\tau(f,g)(\cdot-c_\tau(a^{(n)},b^{(n)}))\|_{L^p(\Omega)}.
	\]
    We claim that there is at most one $\tau_* \in [0,1]$ such that $\{c_{\tau_*}(a^{(n)},b^{(n)})\}_n$ is bounded --- as a result, for any fixed $\tau\neq\tau_*$ we have $|c_\tau(a^{(n)},b^{(n)})|\to\infty$ as $n\to\infty$. Indeed, the following identity follows directly from the definition of \(c_\tau\):	
	\begin{equation}\label{eq:center-diff}
		c_{\tau_1}(a,b)-c_{\tau_2}(a,b)
		=(\tau_1-\tau_2)\,\big(x_b-x_a,\ -(\xi_b-\xi_a)\big)
		=(\tau_1-\tau_2)\,R\,(b-a),
	\end{equation}
	with $R=\mathrm{diag}(I_x,-I_\xi)$ invertible. Consequently, if two distinct $\tau_1\ne\tau_2$ both yield bounded center sequences along a subsequence, then \eqref{eq:center-diff} implies $(\tau_1-\tau_2)\,R\,(b^{(n)}-a^{(n)})$ is bounded, contradicting $|a^{(n)}-b^{(n)}|\to\infty$.

    Therefore, for any fixed $\tau\neq\tau_*$ we have $|c_\tau(a^{(n)},b^{(n)})|\to\infty$. The standard argument from the proof of Theorem~\ref{thm-avgbd} thus shows that \[
	\|\mathcal{B}_n(\tau;\cdot)\|_{L^p(\Omega)}\xrightarrow[n\to\infty]{}0
	\qquad\text{for each fixed }\tau\neq\tau_*.
	\]
	We wish to invoke dominated convergence with respect to $\tau$, so that 
		\[
	\|\BJ(\pi(a^{(n)})f,\pi(b^{(n)})g)\|_{L^p(\Omega)}
	\le \int_0^1 \|\mathcal{B}_n(\tau;\cdot)\|_{L^p(\Omega)}\, \dd \tau
	\ \xrightarrow[n\to\infty]{}\ 0.
	\]
	All we need is to find $h \in L^1(0,1)$ such that
	\(
	\|\mathcal{B}_n(\tau;\cdot)\|_{L^p(\Omega)}\lesssim h(\tau) 
	\)
	for all $\tau\in(0,1)$ and all $n$. To this aim, when $1 \le p \le 2$ we can resort to the embedding $L^2(\Omega) \subset L^p(\Omega)$, H\"older's inequality and \cite[Theorem 4.14]{CT_20} to obtain
	\begin{equation}
		\|\mathcal{B}_n(\tau;\cdot)\|_{L^p(\Omega)} \lesssim |\Omega|^{1/p-1/2}	\|\mathcal{B}_n(\tau;\cdot)\|_{L^2} \lesssim |\Omega|^{1/p-1/2} \|f\|_{L^2}\|g\|_{L^2}. 
	\end{equation}
	When $p>2$ we simply interpolate: 
	\begin{align}
		\| \mathcal{B}_n(\tau;\cdot) \|_{L^p(\Omega)} & \le \| \mathcal{B}_n(\tau;\cdot) \|_{L^2(\Omega)}^{2/p}
		\| \mathcal{B}_n(\tau;\cdot) \|_{L^\infty(\Omega)}^{1-2/p} \\
		& \lesssim (\tau(1-\tau))^{-\frac{1}{2}(1-\frac{2}{p})}\|f\|_{L^2} \|g\|_{L^2},
	\end{align} and the resulting function is summable on $(0,1)$ for every $p>2$. The analogous inequality follows from the identity \(\BJ(f,g)=\overline{\BJ(g,f)}\).
\end{proof}

We can now prove the continuity result for $\BJ$.

\begin{proof}[Proof of Proposition~\ref{prop-BJ-wusc}]
Set $h^{(n)}\coloneqq f^{(n)}-f$, so $h^{(n)}\rightharpoonup 0$ in $L^2(\R)$. By bilinearity,
\begin{equation}\label{eq-BJ-split}
\BJ(f^{(n)})=\BJ(f)+\BJ(h^{(n)})+2\Re(\BJ(f,h^{(n)})).
\end{equation}
It is therefore enough to prove that
\begin{equation}\label{eq:two-vanish}
\|\BJ(h^{(n)})\|_{L^p(\Omega)}\to 0
\qquad \text{and} \qquad
\|\BJ(f,h^{(n)})\|_{L^p(\Omega)}\to 0.
\end{equation}
Since \(h^{(n)}\) is a bounded sequence, we can without loss of generality assume \(\|h^{(n)}\|_{L^2}\leq 1\) for all \(n\). 

\smallskip
\noindent \textbf{Step 1.} \textit{Profile decomposition of $h^{(n)}$ and absence of compact profiles.}
We invoke Lemma~\ref{lem-profdec} and apply it to the bounded sequence $h^{(n)}$: For each fixed $k\in\N$ we can write
\[
h^{(n)}=\sum_{j=1}^k \pi(z_j^{(n)}) \phi_j + w_k^{(n)}, 
\]
with
\begin{align*}
|\zjn-\zjpn|&\to\infty\ \  (j\ne j'), &
\sum_{j=1}^\infty\|\phi_j\|_{L^2}^2&\le 1, &
\lim_{k\to\infty} \limsup_{n\to\infty}\|w_k^{(n)}\|_{M^\infty}&=0.
\end{align*}
Since $h^{(n)}\rightharpoonup 0$, no compact profile can occur: If some $z_{j_0}^{(n)}$ were bounded in $\rdd$, then $\pi(z_{j_0}^{(n)})^*h^{(n)}\rightharpoonup \phi_{j_0}\ne 0$, contradicting $h^{(n)}\rightharpoonup 0$. Hence
\begin{equation}\label{eq-all-escape}
|z_j^{(n)}| \to \infty \quad \text{for every } j\ge 1.
\end{equation}

\smallskip
\noindent \textbf{Step 2.} \textit{Vanishing of the profile contributions.}
By covariance, inner regularity and the fact that $\BJ\phi_j\in C_0(\R^2)$ for $d=1$ (see, e.g., \cite[Example~6.1]{Tauberian}), for each fixed $j$ we infer 
\[
\|\BJ(\pi(z_j^{(n)})\phi_j)\|_{L^p(\Omega)}
=\|\BJ\phi_j(\cdot-z_j^{(n)})\|_{L^p(\Omega)} \to 0. 
\]
Moreover, for $j\ne j'$ we have $|\zjn -\zjpn|\to\infty$, hence Lemma~\ref{lem:BJ-vanish} yields for every $j\ne j'$
\begin{equation}\label{eq:cross-vanish}
\lim_{n\to\infty}\|\BJ(\pi(z_j^{(n)})\phi_j,\pi(z_{j'}^{(n)})\phi_{j'})\|_{L^p(\Omega)}=0.
\end{equation}
Summing over $j=1,\dots,k$ and then letting $k$ be fixed, we therefore get
\begin{equation}\label{eq-pure-vanish}
\lim_{n\to\infty} \Big\|\BJ\Big(\sum_{j=1}^k \pi(z_j^{(n)})\phi_j\Big)\Big\|_{L^p(\Omega)}=0.
\end{equation}

\smallskip
\noindent \textbf{Step 3.} \textit{Remainder terms.} 
Fix a compact set $K\subset\R^2$ with $|\Omega\setminus K|$ arbitrarily small. The terms accounting for the remainders $\wkn$ vanish by an argument similar to the one already used before, since by Theorem~\ref{thm-embedFull} we have
\begin{align} 
	\| \BJ(f,g) \|_{L^2(K)} &\lesssim_{K} \int_0^1 \| W_\tau(f,g) \|_{W( L^2,L^\infty)} \dd \tau \\ & \lesssim \Big(  \int_0^1 (\tau(1-\tau))^{-1/2} \dd \tau \Big) \|f\|_{L^2} \|g \|_{M^\infty}\\
    &\lesssim \|f\|_{L^2} \|g \|_{M^\infty}.
\end{align}
As such, arguing as in the proof of Theorem~\ref{BigTheorem} this bound is sufficient to show that
	\[
	\lim_{k\to\infty}\limsup_{n\to\infty}\ \|\BJ(\Fkn,\wkn)\|_{L^p(\Omega)}=0,
	\]
	and analogously for $\BJ(\wkn,\Fkn)$ and $\BJ(\wkn)$.

\smallskip
\noindent \textbf{Step 4.} \textit{Conclusion.}
Expanding $\BJ(h^{(n)})$ by sesquilinearity into pure, cross and remainder terms, then combining the results in the previous steps, after letting  $n\to\infty$ first and then $k\to\infty$ (which we can do by the same reasoning as in the proof of Theorem~\ref{BigTheorem}), we obtain
\[
\lim_{n \to \infty} \|\BJ(h^{(n)})\|_{L^p(\Omega)}=0.
\]
The same argument allows us to deal with $\BJ(f,h^{(n)})$, where one can resort to Lemma~\ref{lem:BJ-vanish} with the fixed packet $f$ (i.e., take $a^{(n)}=\zjn$, $b^{(n)}=0$ so that $|a^{(n)}-b^{(n)}|\to\infty$) to get
\[
\lim_{n\to\infty}\|\BJ(f,\pi(z_j^{(n)})\phi_j)\|_{L^p(\Omega)}=0
\qquad\text{for each fixed }j.
\]
Therefore, the claim $\|\BJ f^{(n)}-\BJ f\|_{L^p(\Omega)}\to 0$ is proved. 
\end{proof}

We conclude this section by detailing the behavior of the $L^\infty$ concentration problem. 
\begin{prop}
            Let \(\Omega\subset \R^{2}\) be measurable such that \(0<|\Omega|<\infty\). Then
    \begin{align}
        \sup_{f\in L^2(\R)\setminus\{0\}}\frac{\|\BJ f \|_{L^\infty(\Omega)}}{\|f\|_{L^2}^2}=\pi,
    \end{align}
    but the supremum is not attained.
    \end{prop}
    \begin{proof}
            The Born--Jordan distribution is covariant and when \(d=1\) we have \(\BJ f\in C_0(\R^{2})\) for all \(f\in L^2(\R)\) (see~\cite[Example 6.1]{Tauberian}). Therefore, we mimic the proof (and notation) of Proposition~\ref{prop:tauinfty}, so without loss of generality we can assume that the \(\infty\)-norm is attained at \(z=0\) and we need to compute 
            \[\sup_{f\in L^2(\R)\setminus\{0\}}\frac{|\BJ f(0)|}{\|f\|_{L^2}^2}. \]
Now, with the substitution $s=\tfrac{\tau}{1-\tau}$ we have 
        \begin{equation} \label{eq-bja}
\BJ f(0)=\int_0^1 (\tau(1-\tau))^{-1/2}\,\big\langle D_{-\frac{\tau}{1-\tau}} f,f\big\rangle\, \dd \tau = \int_0^\infty \frac{s^{-1/2}}{1+s}\,\langle D_{-s} f,f\rangle\, \dd s. \end{equation} The upper bound $\abs{\BJ f(0)} \le \pi \|f\|_{L^2}^2$ is then clear. 

If $Rf(y)=f(-y)$ denotes the reflection operator, after setting $s=e^t$ with $t \in \R$ and introducing the strongly continuous unitary group $G_t = D_{e^t}$, we can write 
\[ \BJ f(0) = \int_\R \frac{1}{2\cosh(t/2)}\, \langle G_t R f,f \rangle\, \dd t  = \langle Sf,f\rangle, \qquad f \in L^2(\R), \] where we introduced the self-adjoint operator-valued integral
\[ S \coloneqq \int_\R k(t) G_t R \ \dd t, \qquad k(t)=\frac{1}{2\cosh(t/2)},\] 
which is well defined since $\norm{S}_{L^2 \to L^2}\le \norm{k}_{L^1} < \infty$.

Since $R$ commutes with $G_t$, $S$ leaves the subspaces of odd and even functions invariant. If $T$ denotes the restriction to $L^2_e(\R)$, the restriction to odd functions coincides with $-T$, hence $\norm{S}_{L^2 \to L^2}=\max\{\norm{T}_{L^2_e \to L^2},\norm{-T}_{L^2_o \to L^2}\}=\norm{T}_{L^2_e \to L^2}$. In particular, we have
\[ \BJ f(0) = \int_\R \frac{1}{2\cosh(t/2)}\, \langle G_t f,f \rangle\, \dd t  = \langle Tf,f\rangle, \qquad f \in L^2_e(\R). \]
Let us now introduce the unitary operator $U \colon L^2_e(\R)\to L^2(\R)$ to be 
\[
(Uf)(y)=\sqrt{2} \ e^{y/2}f(e^y), 
\]
satisfying $UG_tU^{-1}=T_t$, that is $G_t$ is unitarily equivalent to the translation group $(T_t f)(y)=f(y+t)$. Since $U$ is unitary, $G_t$ strongly continuous and $k \in L^1(\R)$, we may pass $U$ through the Bochner integral and obtain
\[
UTU^{-1}f=\int_{\R} k(t)T_t f \ \dd t=(k*f),
\]
hence $UTU^{-1}$ is the convolution operator with kernel $k$ on $L^2(\R)$. As such, we obtain by Plancherel's theorem
\[ \|UTU^{-1}f\|_{L^2}=\|k*f\|_{L^2}=\|\widehat{k}\widehat{f}\|_{L^2}\le \|\widehat{k}\|_{L^\infty}\|f\|_{L^2}.
\]  An easy computation shows that $\widehat{k}(\xi)=\pi \ \mathrm{sech}(2\pi^2\xi)$, whence \[\|T\|_{L^2_e \to L^2}=\|UTU^{-1}\|_{L^2 \to L^2}=\|\widehat{k}\|_\infty = \widehat{k}(0)= \pi.\] 

Non-attainment on even functions follows since equality in H\"older's inequality $\|\widehat{k} \widehat{f}\|_2\le \|\widehat{k}\|_\infty\|\widehat{f}\|_2$ with $f \in L^2_e(\R)\setminus\{0\}$ would force $|\widehat{k}|=\|\widehat{k}\|_\infty$ a.e. on the support of $\widehat{f}$, or equivalently $\widehat{f}$ to be supported where $|\widehat{k}|=\pi$, i.e.\ at $\{\xi=0\}$, a null set, leading to $\widehat{f}=0$. Since $S=-T$ on odd functions, non-attainment on the whole $L^2(\R)$ follows. 
\end{proof}

We would ideally like to extend our existence result for the Born--Jordan distribution to any dimension \(d\). However, the Born--Jordan distribution decays significantly worse when \(d>1\). For instance, the integral $\int_0^1 C(\tau) \dd \tau$ (cf.\ Theorem~\ref{thm-embedFull}) needed to control the error terms in the profile decomposition diverges when \(d>1\). Similarly, the proof of Lemma~\ref{lem:BJ-vanish} extends to higher dimension with a tradeoff between $p$ and $d$: The claim still holds for all $1 \le p < p_*(d)$, where
	\begin{equation}
		p_*(d) \coloneqq \begin{cases}
			\infty & (d=1,2) \\
			\tfrac{2d}{d-2} & (d\ge 3).
		\end{cases}
    \end{equation}
Therefore, the contribution from the cross terms is \(0\) when \(d<3\), but for \(d\geq 3\) the cross-term contribution is currently guaranteed to be \(0\) only provided that \(p<p_*(d)\). As such, there is reason to believe that the Born--Jordan problem is only well-posed when \(d=1\).

\begin{question} Is \(L_{\mathrm{BJ}}=\infty\) when $d>1$? What happens if we restrict to \(f\in M^q(\rd)\) with $1 \le q <2$ instead?
\end{question}

\section*{Acknowledgments}

It is a pleasure to thank Franz Luef for discussions on the topics of this note. We also thank José Luis Romero for stimulating remarks on a first draft which led us to better highlight some results (namely, Propositions~\ref{prop:wuSCfail} and~\ref{prop-BJ-wusc}). 

The second author thanks Politecnico di Torino for its hospitality while on a visit in Spring 2025 during which the work started.

The third author is a member of Gruppo Nazionale per l'Analisi Matematica, la Probabilit\`a e le loro Applicazioni (GNAMPA) --- Istituto Nazionale di Alta Matematica (INdAM). 

\bibliographystyle{elsarticle-num}

\end{document}